\documentclass[12pt,twoside]{article}

\usepackage{fancyhdr}
\usepackage{amsmath}
\usepackage{amssymb}
\usepackage{amsfonts}
\usepackage{euscript}
\usepackage{oldgerm}
\usepackage{calligra}
\usepackage{mathrsfs}
\usepackage{hyperref}
\usepackage{url}
\usepackage{path}
\usepackage{ecltree}
\usepackage{epsfig}
\usepackage{lastpage}
\usepackage{epic}

\usepackage{multirow}
\usepackage{graphicx}
\usepackage{stmaryrd}
\usepackage{wasysym}
\usepackage{pifont}

\usepackage{draftwatermark}
\SetWatermarkAngle{45}
\SetWatermarkLightness{0.9}
\SetWatermarkFontSize{5cm}
\SetWatermarkScale{2}
\SetWatermarkText{ }

\renewcommand{\thefootnote}{\fnsymbol{footnote}}

\long\def\sfootnote[#1]#2{\begingroup
\def\thefootnote{\fnsymbol{footnote}}\footnote[#1]{#2}\endgroup}

\newtheorem{theorem}{Theorem}
\newtheorem{definition}[theorem]{Definition}
\newtheorem{notation}[theorem]{Notation}
\newtheorem{convention}[theorem]{Convention}
\newtheorem{lemma}[theorem]{Lemma}
\newtheorem{corollary}[theorem]{Corollary}

\newtheorem{example}[theorem]{Example}
\newtheorem{conjecture}[theorem]{Conjecture}
\newtheorem{question}[theorem]{Question}
\newtheorem{remark}[theorem]{Remark}

\newenvironment{proof}{\noindent\mbox{\bf Proof.}}
{\hfill\mbox{$\Subset\!\!\!\!\Supset$}\bigskip}

\begin{document}
\pagestyle{fancy}
\lhead[page \thepage \ (of \pageref{LastPage})]{}
\chead[{\bf Herbrand Consistency  of  Some Arithmetical Theories}]{{\bf Herbrand Consistency  of  Some Arithmetical Theories}}
\rhead[]{page \thepage \ (of \pageref{LastPage})}
\lfoot[\copyright\ {\sf Saeed Salehi 2010}]{$\varoint^{\Sigma\alpha\epsilon\epsilon\partial}_{\Sigma\alpha\ell\epsilon\hslash\imath}\centerdot${\footnotesize {\rm ir}}}
\cfoot[{\footnotesize {\tt http:\!/\!/saeedsalehi.ir/}}]{{\footnotesize {\tt http:\!/\!/saeedsalehi.ir/}}}
\rfoot[$\varoint^{\Sigma\alpha\epsilon\epsilon\partial}_{\Sigma\alpha\ell\epsilon\hslash\imath}\centerdot${\footnotesize {\rm ir}}]{\copyright\ {\sf Saeed Salehi 2010}}
\renewcommand{\headrulewidth}{1pt}
\renewcommand{\footrulewidth}{1pt}
\thispagestyle{empty}

\begin{center}
\begin{table}
\begin{tabular}{| c | l  || l | c |}
\hline
 \multirow{7}{*}{\includegraphics[scale=0.75]{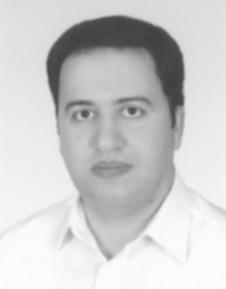}}&    &  &
 \multirow{7}{*}{ \ \ ${\huge \varoint^{\Sigma\alpha\epsilon\epsilon\partial}_{\Sigma\alpha\ell\epsilon\hslash\imath}\centerdot}${{\rm ir}}} \ \ \ \ \\
 &     \ \ {\large{\sc Saeed Salehi}}  \ \  \ & \ \    Tel: \, +98 (0)411 339 2905  \ \  &  \\
 &   \ \ Department of Mathematics \ \  \ & \ \ Fax: \ +98 (0)411 334 2102 \ \  & \\
 &   \ \ University of Tabriz \ \ \  & \ \ E-mail: \!\!{\tt /root}{\sf @}{\tt SaeedSalehi.ir/} \ \  &  \\
 &  \ \ P.O.Box 51666--17766 \ \ \ &   \ \ \ \ {\tt /SalehiPour}{\sf @}{\tt TabrizU.ac.ir/} \ \  &  \\
 &   \ \ Tabriz, Iran \ \ \ & \ \ Web: \  \ {\tt http:\!/\!/SaeedSalehi.ir/} \ \ &  \\
 &    &  &  \\
 \hline
\end{tabular}
\end{table}
\end{center}


\vspace{2em}

\begin{center}
{\bf {\Large  Herbrand Consistency  of  Some Arithmetical Theories
}}
\end{center}

\vspace{2em}
\begin{abstract}
G\"odel's second incompleteness theorem is proved for Herbrand consistency of some arithmetical theories
with bounded induction, by using a technique of logarithmic shrinking the witnesses of bounded formulas, due to Z. Adamowicz [Herbrand consistency and bounded arithmetic,  \textit{Fundamenta Mathematicae} 171 (2002) 279--292]. In that paper, it was shown that one cannot always shrink the witness of a bounded formula logarithmically, but in the presence of Herbrand consistency, for theories ${\rm I\Delta_0+\Omega_m}$ with $m\geqslant 2$, any witness for any bounded formula can be shortened logarithmically. This immediately implies the unprovability of Herbrand consistency of a theory $T\supseteq {\rm I\Delta_0+\Omega_2}$ in $T$ itself.

In this paper, the above results are generalized for ${\rm I\Delta_0+\Omega_1}$. Also after tailoring the definition of Herbrand consistency for ${\rm I\Delta_0}$ we prove the corresponding theorems for ${\rm I\Delta_0}$. Thus the Herbrand version of G\"odel's second incompleteness theorem follows  for the theories ${\rm I\Delta_0+\Omega_1}$ and ${\rm I\Delta_0}$.

\bigskip

\centerline{${\backsim\!\backsim\!\backsim\!\backsim\!\backsim\!\backsim\!\backsim\!
\backsim\!\backsim\!\backsim\!\backsim\!\backsim\!\backsim\!\backsim\!
\backsim\!\backsim\!\backsim\!\backsim\!\backsim\!\backsim\!\backsim\!
\backsim\!\backsim\!\backsim\!\backsim\!\backsim\!\backsim\!\backsim\!
\backsim\!\backsim\!\backsim\!\backsim\!\backsim\!\backsim\!\backsim\!
\backsim\!\backsim\!\backsim\!\backsim\!\backsim\!\backsim\!\backsim\!
\backsim\!\backsim\!\backsim\!\backsim\!\backsim\!\backsim\!\backsim\!
\backsim\!\backsim\!\backsim\!\backsim\!\backsim\!\backsim\!\backsim\!
\backsim\!\backsim\!\backsim\!\backsim\!\backsim\!\backsim\!\backsim\!
\backsim\!\backsim\!\backsim\!\backsim\!\backsim\!\backsim\!\backsim\!
\backsim\!\backsim\!\backsim}$}

\bigskip



\end{abstract}


\bigskip

\vfill

\begin{center}
\begin{tabular}{| c |}
\hline
  \\
 {\sc Saeed Salehi}, {\bf Herbrand Consistency  of  Some Arithmetical Theories},  \textmd{Manuscript 2010}.  \\
\\
 {\large {\tt http:\!/\!/saeedsalehi.ir/ 
  }}
     \qquad  \qquad
   {\sf Status}: 
 {\sc MANUSCRIPT ({Submitted})} 
  \\
  \\
   \hline
\end{tabular}
\end{center}
\vspace{-1em}
\hspace{.75em} \textsl{\footnotesize Date: 08 June 2010}

\vfill

\bigskip

\centerline{page 1 (of \pageref{LastPage})}


\newpage
\setcounter{page}{2}
\SetWatermarkAngle{55}
\SetWatermarkLightness{0.9}
\SetWatermarkFontSize{50cm}
\SetWatermarkScale{2.25}
\SetWatermarkText{\!\!{\sc MANUSCRIPT (Submitted)}}

\section{Introduction}
By G\"odel's first incompleteness theorem ${\sf Truth}$ is not the same as ${\sf Provability}$ in sufficiently strong theories. In other words,  ${\sf Provable}$ is a proper subset of ${\sf True}$, and thus ${\sf True}$ is not conservative over ${\sf Provable}$. It is not even $\Pi_1-$conservative; i.e., there exists a $\Pi_1-$formula, in theories which can interpret enough arithmetic, which is true but unprovable in those theories. Thus one way of comparing the strength of a theory $T$ over one of its sub-theories $S$ is considering the $\Pi_1-$conservativeness of $T$ over $S$. And G\"odel's second incompleteness theorem provides such a $\Pi_1-$candidate: ${\rm Con}(S)$, the statement of the consistency of $S$. By that theorem $S\not\vdash {\rm Con}(S)$, but if $T\vdash {\rm Con}(S)$ then $T$ is not $\Pi_1-$conservative over $S$.

Examples abound in mathematics and logic: Zermelo-Frankel Set Theory ${\rm ZFC}$ is not $\Pi_1-$conservative over Peano's Arithmetic ${\rm PA}$, because ${\rm ZFC}\vdash {\rm Con}({\rm PA})$ but
${\rm PA}\not\vdash {\rm Con}({\rm PA})$. Inside ${\rm PA}$ the $\Sigma_n-$hierarchy is not a $\Pi_1-$conservative hierarchy, since ${\rm I\Sigma_{n+1}}\vdash{\rm Con}({\rm I\Sigma_n})$ though ${\rm I\Sigma_{n}}\not\vdash{\rm Con}({\rm I\Sigma_n})$; see e.g. \cite{HP98}. Then below the  theory ${\rm I\Sigma_1}$ things get more complicated: for $\Pi_1-$separating ${\rm I\Delta_0}+{\rm Exp}$ over ${\rm I\Delta_0}$ the candidate ${\rm Con}({\rm I\Delta_0})$ does not work, because ${\rm I\Delta_0}+{\rm Exp}\not\vdash{\rm Con}({\rm I\Delta_0})$. For this $\Pi_1-$separation, Paris and Wilkie \cite{PaWi81} suggested the notion of cut-free consistency instead of usual - Hilbert style - consistency predicate. Here one can show that ${\rm I\Delta_0}+{\rm Exp}\vdash{\rm CFCon}({\rm I\Delta_0})$, and then it was presumed that ${\rm I\Delta_0}\not\vdash {\rm CFCon}({\rm I\Delta_0})$, where ${\rm CFCon}$ stands for cut-free consistency. But this presumption took a rather long time to be established. Meanwhile, Pudl\'ak in \cite{Pud85} established the $\Pi_1-$separation of ${\rm I\Delta_0}+{\rm Exp}$ over ${\rm I\Delta_0}$  by other methods, and mentioned the unprovability of ${\rm CFCon}({\rm I\Delta_0})$ in ${\rm I\Delta_0}$ as an open problem.  This problem is interesting in its own right.  Indeed G\"odel's second incompleteness theorem has been generalized to all consistent theories containing Robinson's Arithmetic ${\rm Q}$, in the case of Hilbert consistency; see \cite{HP98}. But for cut-free consistency it is still an open problem whether the theorem holds for ${\rm Q}$, and its not too strong extensions. This is a double strengthening  of G\"odel's second incompleteness theorem: weakening the theory and weakening the consistency predicate. Let us note that since cut-free provability is stronger than usual Hilbert provability (with a super-exponential cost), then cut free consistency is  a weaker notion of consistency. Indeed, proving G\"odel's second incompleteness theorem for weak notions of consistencies in weak arithmetics turns out to be a difficult problem. We do not intend here to give a thorough history of this ongoing research area, let us just mention a few  results:

\noindent $\bullet \ $ Z. Adamowicz was the first one to demonstrate  the unprovability of cut free consistency in bounded arithmetics, by  proving in an unpublished manuscript in 1999 (later appeared as a technical report \cite{Ada96}) that the tableau consistency of ${\rm I\Delta_0+\Omega_1}$ is not provable in itself. Later  with P. Zbierski (2001) she proved G\"odel's second incompleteness theorem for Herbrand consistency of ${\rm I\Delta_0}+{\rm \Omega_2}$ (see \cite{Ada01}), and a bit later she gave a model theoretic proof of it in 2002; see \cite{Ada02}.

\noindent $\bullet \ $ D. E. Willard  introduced an ${\rm I\Delta_0}-$provable $\Pi_1-$formula $V$ and showed that any theory whose axioms contains $Q+V$ cannot prove its own tableaux consistency. He also showed that tableaux consistency of ${\rm I\Delta_0}$ is not provable in itself, see \cite{Wil02,Wil07}; this proved the conjecture of Paris and  Wilkie mentioned above.

\noindent $\bullet \ $ S. Salehi (see \cite{Sal02} Chapter 3 and also \cite{Sal01}) showed  the unprovability of Herbrand consistency of a re-axiomatization of ${\rm I\Delta_0}$ in itself, the proof of which was heavily based on \cite{Ada01}. The re-axiomatization used ${\rm PA^-}$, the theory of the positive fragment of a discretely ordered ring, as the base theory, instead of ${\rm Q}$, and assumed two ${\rm I\Delta_0}-$derivable sentences as axioms. Also the model-theoretic proof of Z. Adamowicz in \cite{Ada02} was generalized to the ${\rm I\Delta_0+\Omega_1}$ case in Chapter 5 of \cite{Sal02}. A polished and updated proof of it appears in the present paper.

\noindent $\bullet  \ $  L. A. Ko{\l}odziejczyk showed in \cite{Kol06} that the notion of Herbrand consistency cannot $\Pi_1-$separate the hierarchy of bounded arithmetics (this $\Pi_1-$separation is still an open problem). Main results are the existence of an $n$ for any given $m\geqslant 3$ such that $S_m\not\vdash{\rm HCon}(S_m^n)$, and the existence of a natrual $n$ such that $\bigcup_mS_m\not\vdash{\rm HCon}(S_3^n)$, where ${\rm HCon}$ stands for Herbrand consistency.

\noindent $\bullet  \ $  Z. Adamowicz and K. Zdanowski have obtained some results on the unprovability of the relativized notion of Herbrand consistency in theories containing ${\rm I\Delta_0+\Omega_1}$; see \cite{AZ07}. Their paper contains some insightful ideas about the notion of Herbrand consistency.

For ${\rm I\Delta_0+\Omega_1}$ the arguments are rather smoother, in comparison  to the case of ${\rm I\Delta_0}$. Our proof for the main theorem on ${\rm I\Delta_0+\Omega_1}$ borrows many ideas from \cite{Ada02}, the major difference being the coding techniques and making use of a more liberal definition of Herbrand consistency. The definition of ${\rm HCon}$ given in \cite{Ada01} and  \cite{Ada02} depends on a special coding given there.
For reading the present paper no familiarity with \cite{Ada01} is needed, but a theorem of \cite{Ada02} will be of critical use here (Theorem \ref{th2}). We will even use a modified version of it (Theorem \ref{ij1}).  For ${\rm I\Delta_0}$ we will see that our definition of ${\rm HCon}$ is not best suited for this theory; and we will actually tailor it for ${\rm I\Delta_0}$. A hint for the obstacles in tackling Herbrand consistency in ${\rm I\Delta_0}$ can be found in Chapters 3 and 4 of \cite{Sal02}.

In Section 2 we introduce the ingredients of Herbrand's theorem from the scratch, and then explain how they can be arithmetized by G\"odel coding. This sets the stage for Section 3 in which we formalize the notion of Herbrand model and use it to prove our main theorem for ${\rm I\Delta_0+\Omega_1}$. Finally in Section 4 we modify our definitions and theorems to fit the ${\rm I\Delta_0}$ case. After pinpointing the places where we have made an essential use of ${\rm \Omega_1}$, we do some tailoring for ${\rm I\Delta_0}$, and prove our main result for ${\rm I\Delta_0}$. We  finish the paper with some conclusions and  some open questions.


\section{Basic Definitions and Arithmetizations}
This section introduces the notions of Herbrand provability and Herbran consistency, and a way of formalizing and arithmetizing these concepts. The first subsection can be read by any logician. The second subsection gets more technical with G\"odel coding, for which some familiarity with \cite{HP98} is presumed.

\subsection{Herbrand Consistency}
Skolemizing a formula is usually performed on prenex normal forms (see e.g. \cite{Buss95}), and since prenex normalizing a formula is not necessarily done in a unique way, then one may get different Skolemized forms of a formula. For example, the tautology $F\!=\!\forall x \phi(x)\rightarrow \forall x \phi(x)$ can be prenex normalized into either $\forall x\exists y (\phi(y)\rightarrow\phi(x))$ or $\exists y\forall x (\phi(y)\rightarrow\phi(x))$. These two formulas can be Skolemized respectively as $\phi({\mathfrak f}(x))\rightarrow\phi(x)$ and $\phi({\mathfrak c})\rightarrow\phi(x)$, where ${\mathfrak f}$ is a new unary function symbol, and ${\mathfrak c}$ is a new constant symbol.
Here we briefly describe a way of Skolemizing a (not-necessarily prenex normal) formula which results in a somehow unique  (up to a variable renaming) formula.

A formula is in {\em negation normal} form when the implication symbol does not appear in it, and the negation symbol appears in front of atomic formulas only. A formula can be (uniquely) negation normalized  by the following rewriting rules:
\begin{align*}
(A\rightarrow B) & \Longmapsto (\neg A\vee B)  & \neg \neg A & \Longmapsto A \\
\neg (A\vee B) & \Longmapsto (\neg A\wedge\neg B) & \neg (A\wedge B) & \Longmapsto  (\neg A\vee\neg B)\\
\neg \forall x A(x)  & \Longmapsto \exists x \neg A(x) & \neg \exists x A(x) & \Longmapsto \forall x \neg A(x)
\end{align*}
A formula is called {\em rectified} if no variables appears both bound and free in it,
and  different quantifiers refer to different variables. A formula is called {\em rectified negation normal} if it is both negation normalized and rectified. Again, any formula can be rectified. Indeed, any given formula is equivalent to its rectified negation normal form (RNNF) which can be obtained from the formula
in a unique (up to a variable renaming) way (see e.g. \cite{BBR07}).

 Now we introduce Skolem functions for existential formulas: for any
(not necessarily RNNF) 
 formula of the form $\exists x A(x)$, let   ${\mathfrak f}_{\exists x A(x)}$ be a
new $m-$ary function symbol where $m$ is the number of the free variables of
$\exists x A(x)$.  When $m=0$ then  ${\mathfrak f}_{\exists x A(x)}$ will obviously be a new constant symbol (cf. \cite{Buss95}).

 \begin{definition}{\rm
Let $\varphi$ be an RNNF formula. Define $\varphi^{\mathsf S}$   by induction:
\begin{itemize}
\item  $\varphi^{\mathsf S}=\varphi$ for atomic or negated-atomic $\varphi$;

\item $(\varphi\circ\psi)^{\mathsf S}=\varphi^{\mathsf S}\circ\psi^{\mathsf S}$
for $\circ\!\in\!\{\wedge,\vee\}$ and RNNF formulas $\varphi,\psi$;

\item $(\forall x\varphi)^{\mathsf S}=\forall x\varphi^{\mathsf S}$;

\item $(\exists x\varphi)^{\mathsf S}=\varphi^{\mathsf S}[{\mathfrak f}_{\exists x \varphi(x)}(\overline{y})/x]$ where $\overline{y}$ are the free
variables of $\exists x \varphi(x)$ and the formula
$\varphi^{\mathsf S}[{\mathfrak f}_{\exists x \varphi(x)}(\overline{y})/x]$
results from the formula $\varphi^{\mathsf S}$ by replacing all the occurrences of
the variable $x$ with the term ${\mathfrak f}_{\exists x \varphi(x)}(\overline{y})$.
\end{itemize}

\noindent The Skolemized form of any (not necessarily RNNF) formula $\psi$ is
obtained in the following way: using the above rewriting rules we negation normalize $\psi$ and then rename the
repetitive variables (if any) to get a rectified negation normal form of $\psi$, say $\varphi$.
Then we get  $\varphi^{\mathsf S}$ by the above definition, and remove all the
 (universal) quantifiers in it (together with the variables next to them).
We denote thus resulted Skolemized form of $\psi$ by $\psi^{\rm Sk}$.\hfill $\subset\!\!\!\!\supset$}
\end{definition}

Note that $\psi^{\rm Sk}$ can be obtained from $\psi$ in a unique
(up to a variable renaming) way, and it is an open (quantifier-less) formula. For the above example $F$, assuming that $\phi$ is atomic, we get
\newline\centerline{$F^{\mathsf S}=(\exists x\neg\phi(x)\vee\forall x\phi(x))^{\mathsf S}=\neg\phi({\mathfrak c})\vee\forall x\phi(x),$} and thus $F^{\rm Sk}=\neg\phi({\mathfrak c})\vee\phi(x)\equiv \phi({\mathfrak c})\rightarrow\phi(x).$

\begin{definition}{\rm
An Skolem instance of the formula $\psi$ is any formula resulted
from substituting the free variables of $\psi^{\rm Sk}$ with some
terms. So, if $x_1,\ldots,x_n$ are the free variables of
$\psi^{\rm Sk}$ (thus written as $\psi^{\rm
Sk}(x_1,\ldots,x_n)$) then an Skolem instance of $\psi$ is
$\psi^{\rm Sk}[t_1/x_1,\cdots,t_n/x_n]$ where $t_1,\ldots,t_n$
are terms (which could be constructed from the Skolem functions
symbols).

\noindent  Skolemized form of a theory
$T$ is by defintion $T^{\rm Sk}=\{\varphi^{\rm Sk}\mid\varphi\!\in\! T\}$. \hfill
$\subset\!\!\!\!\supset$}\end{definition}

We are now ready to state an important theorem
discovered by  Herbrand  (probably by also Skolem and G\"odel). This
theorem has got some few names, and by now is a classical theorem in
Mathematical Logic. Here we state a version of the theorem which we will need in the paper.
The proof is omitted,
though it is not too difficult to prove it directly (see e.g. \cite{BBR07}).

\begin{theorem}[Herbrand]{
Any theory $T$ is
equiconsistent with its Skolemized theory $T^{\rm Sk}$. In other
words, $T$ is consistent if and only if every finite set of Skolem
instances of $T$ is (propositionally) satisfiable.\hfill
$\subset\!\!\!\!\supset$
}\end{theorem}
%
%
%
We will use the above theorem, which reduces the consistency of a  first-order theory
to the satisfiability of a propositional theory, for the definition of Herbrand
Consistency: a theory $T$ is Herbrand consistent when every finite set of
Skolem instances of $T$ is propositionally satisfiable. One other
concept is needed for formalizing Herbrand consistency of
arithmetical theories: {\em evaluation}.
\begin{convention}{\rm
Throughout the paper we deal with closed (or ground)
terms (i.e., terms with no variable) and for simplicity we call them
``term". For this to make sense, we may assume that the language of
the theory under consideration has at least one constant symbol.
\hfill $\subset\!\!\!\!\supset$
}\end{convention}
\begin{definition}\label{defeval}{\rm An {\em evaluation} is a function whose domain is the set of all
atomic formulas constructed from a given set of terms $\Lambda$ and
its  range is the set $\{0,1\}$ such that

 (i) \ $p \, [t\!=\!t]=1$ for all $t\!\in\!\Lambda$; and for any terms
$t,s\!\in\!\Lambda$,

 (ii) \  if $p\, [t\!=\!s]=1$ then $p\, [\varphi(t)]=p\,
[\varphi(s)]$ for any atomic formula $\varphi(x)$.

\noindent The relation $\backsim_p$ on $\Lambda$ is defined by
$t\backsim_p s \iff p[t=s]=1$ for $t,s\!\in\!\Lambda$. \hfill $\subset\!\!\!\!\supset$}
\end{definition}
\begin{lemma}
The  relation
$\backsim_p$  defined above is an equivalence relation.
\end{lemma}
\begin{proof}
For $\varphi(x) \equiv (s\!=\!x)$ from $p\,[t\!=\!s]=1$ one can
infer $p\,[s\!=\!t]=p\,[\varphi(t)]=p\,[\varphi(s)]=p\,[s\!=\!s]=1$.
So, $t\backsim_ps$ implies $s\backsim_pt$. Also for $\phi(x)\equiv
(t\!=\!x)$ the condition $p\,[s\!=\!r]=1$ implies
$p\,[t\!=\!s]=p\,[\phi(s)]=p\,[\phi(r)]=p\,[t\!=\!r]$. So, $\backsim_p$ is a symmetric and transitive (also, by definition, a reflexive) relation.
\end{proof}
\begin{notation}{\rm
The $\backsim_p\!\!-$class of a term
$t$ is denoted by $t/p$; and the set of all such $p-$classes for
each $t\!\in\!\Lambda$ is denoted by $\Lambda/p$.

For simplicity, we write $p\models\varphi$ instead of
$p\,[\varphi]=1$; thus $p\not\models\varphi$ stands for
$p\,[\varphi]=0$. This definition of {\em satisfying} can be
generalized to other open formulas in the usual way:
\begin{itemize}
\item $p\models\varphi\wedge\psi$ if and only if $p\models\varphi$ and
$p\models\psi$;
\item $p\models\varphi\vee\psi$ if and only if $p\models\varphi$ or
$p\models\psi$;
\item $p\models\neg\varphi$ if and only if $p\not\models\varphi$.
\hfill$\subset\!\!\!\!\supset$
\end{itemize}
}\end{notation}
Let us note that $\backsim_p$ is a congruence relation as well. That is, for any set of terms $t_i$ and $s_i$ ($i=1,\ldots,n$) and function symobl $f$, if $p\models t_1=s_1\wedge \cdots\wedge t_n=s_n$ then $p\models f(t_1,\ldots,t_n)=f(s_1,\ldots,s_n)$.
\begin{definition}{\rm  If all terms appearing in an Skolem instance of $\phi$
belong to the set $\Lambda$, that formula is called an  Skolem
instance of $\phi$ {\em available} in $\Lambda$.

\noindent An evaluation defined on $\Lambda$ is called a {\em
$\phi-$evaluation} if it satisfies all the Skolem instances of
$\phi$ which are available in $\Lambda$.

\noindent Similarly, for a theory $T$,  a {\em $T-$evaluation} on
$\Lambda$ is an evaluation on $\Lambda$ which satisfies every Skolem
instance of every formula of $T$ which is available in
$\Lambda$.\hfill $\subset\!\!\!\!\supset$}
\end{definition}

For illustrating the above concepts we now present an example.

\begin{example}\label{example1}{\rm
Take the language ${\mathcal L}=\{g,P,R,S\}$ in which $g$ is a binary
function symbol, and $P$ is a binary predicate symbol, and $R,S$ are
unary predicate symbols. Let the theory $T$ be axiomatized by:

$T_1: \ \forall x\exists y P(x,y)$;

$T_2: \ \forall x \big (R(x)\vee S(gx)\big)$;

$T_3: \ \forall x,y \big(\neg P(x,y)\vee\neg S(x)\big)$.

\noindent Let us, for the sake of simplicity, denote
 ${\mathfrak f}_{\exists y P(x,y)}$ by ${\mathfrak f}$; then the Skolemized form of
the above theory is:

$T_1^{\rm Sk}: \  P(x,{\mathfrak f} x)$; \ \ \ $T_2^{\rm Sk}: \  R(x)\vee S(gx)$; \ \ \ $T_3^{\rm Sk}: \  \neg  P(x,y)\vee\neg S(x)$.

\noindent For a constant symbol $c$ let $\Lambda=\{c, gc, {\mathfrak f} c\}$. Then $P(c,{\mathfrak f} c)$ and $R(c)\vee S(gc)$ are Skolem
instances of $T$ (of $T_1$ and $T_2$) available in $\Lambda$, but
the Skolem instance $R(gc)\vee S(ggc)$ of $T_2$ is not available in
$\Lambda$. Let us note also that the Skolem instance $\neg P(gc,{\mathfrak f}
gc)\vee\neg S(gc)$ of $T_3$ is not available in $\Lambda$.

\noindent Let $q$ be an evaluation on $\Lambda$ whose set of true
atomic formulas is $\{P(c,{\mathfrak f} c), R(c)\}$. Then $q$ is a
$T-$evaluation. On the other hand the evaluation $r$  on $\Lambda$
whose set of true atomic formulas is $\{P(c,{\mathfrak f} c), R(c),
S(c)\}$, is not a $T-$evaluation, though it satisfies all the
Skolem instances of $T_1$ and $T_2$ which are available in
$\Lambda$. Note that $r$ does not satisfy the Skolem instance $\neg
P(c,{\mathfrak f} c)\vee\neg S(c)$ of $T_3$. \hfill $\subset\!\!\!\!\supset$
}\end{example}

By the above theorem of  Herbrand, a theory $T$
is consistent if and only if every finite set of its Skolem
instances is satisfiable, if and only if for every finite set of
terms $\Lambda$ there is a $T-$evaluation on $\Lambda$.  And for a
formula $\varphi$, $T\vdash\varphi$ if and only if there exists a
finite set of terms $\Lambda$ such that there is no
$(T+\neg\varphi)-$evaluation on $\Lambda$. We call this notion of
provability, {\em Herbrand Provability}; note that then {\em
Herbrand Consistency} of a theory $T$ means the existence of a
$T-$evaluation on any (finite) set of terms.

\begin{example}\label{example2}{\rm In the previous example, let $\varphi=\forall x
R(x)$. We show  $T\vdash\varphi$ by Herbrand provability. Write
$\neg\varphi=\exists x \neg R(x)$, and let ${{\mathfrak c}}$ denote the Skolem constant
symbol ${\mathfrak f}_{\exists x \neg R(x)}$; so  we have $(\neg\varphi)^{\rm Sk}=\neg R({{\mathfrak c}})$.
Put $\Lambda=\{{{\mathfrak c}}, g{{\mathfrak c}}, {{\mathfrak f}}g{{\mathfrak c}}\}$, and assume (for the sake
of contradiction) that there is a $(T+\neg\varphi)-$evaluation $p$ on
$\Lambda$. Then $p$ must satisfy the following Skolem instances of
$T$ in $\Lambda$: $P(g{{\mathfrak c}},{{\mathfrak f}}g{{\mathfrak c}})$, $R({{\mathfrak c}})\vee S(g{{\mathfrak c}})$,
and $\neg P(g{{\mathfrak c}},{{\mathfrak f}}g{{\mathfrak c}})\vee\neg S(g{{\mathfrak c}})$. Whence $p$ must
also satisfy $\neg S(g{{\mathfrak c}})$ and $R({{\mathfrak c}})$. So $p$ cannot satisfy the
Skolem instance $\neg R({{\mathfrak c}})$ of $\neg\varphi$ in $\Lambda$.
Thus there cannot be any $(T+\neg\varphi)-$evaluation on $\Lambda$;
whence $T\vdash\varphi$.

Note that finding an appropriate $\Lambda$ is as complicated as finding a formal proof.
For example we could not have taken $\Lambda$ as $\{{{\mathfrak c}}, g{{\mathfrak c}}, {{\mathfrak f}} {{\mathfrak c}}\}$,
since the evaluation $q$ in the previous example
would be a $(T+\neg\varphi)-$evaluation on that set.
  \hfill $\subset\!\!\!\!\supset$}\end{example}

The following couple of examples give a thorough illustrations for the above ideas, and they will be actually used later in the paper.

\begin{example}\label{exampleA}{\rm
Let ${\rm Q}$ denote Robinson's Arithmetic over
the language of arithmetic $\langle0,{\mathfrak s},+,\cdot,\leqslant\rangle$, where $0$ is a
constant symbol, ${\mathfrak s}$ is a unary function symbol, $+,\cdot$ are binary
function symbols, and $\leqslant$ is a binary predicate symbol, whose axioms are:
\begin{align*}
A_1: & \ \  \forall x ({\mathfrak s} x\not=0) & A_2: & \ \  \forall x\forall y ({\mathfrak s} x ={\mathfrak s} y\rightarrow x=y)\\
A_3: & \ \   \forall x (x\not=0\rightarrow \exists y[x={\mathfrak s} y]) & \ \   A_4: & \ \   \forall x\forall y (x\leqslant y \leftrightarrow \exists z
[x+z=y])\\
A_5: & \ \   \forall x (x+0=x) & A_6: & \ \   \forall x\forall y (x+{\mathfrak s} y={\mathfrak s} (x+y))\\
A_7:  & \ \   \forall x (x\cdot 0=0) & A_8: & \ \   \forall x\forall y (x\cdot{\mathfrak s} y=x\cdot y + x)
\end{align*}
Let $\psi=\forall x(x\leqslant 0\rightarrow x=0)$ and $\varphi=\forall x\forall y(x\leqslant{\mathfrak s} y\rightarrow x={\mathfrak s} y\vee x\leqslant y)$. We can show
${\rm Q}\vdash\psi$ and ${\rm Q}\vdash\varphi$; these will be proved below by Herbrand provability.
Suppose ${\rm Q}$ has been Skolemized as below:
\begin{align*}
A_1^{\rm Sk}: & \ \   {\mathfrak s} x\not=0 & A_2^{\rm Sk}: & \ \   {\mathfrak s} x\not={\mathfrak s} y\vee x=y\\
A_3^{\rm Sk}: & \ \   x=0 \vee x={\mathfrak s}{\mathfrak p} x & A_4^{\rm Sk}: & \ \   [x\not\leqslant y\vee x+{\mathfrak h} (x,y)=y]\wedge[x+z\not=y\vee x\leqslant y]\\
A_5^{\rm Sk}: & \ \   x+0=x & A_6^{\rm Sk}: & \ \   x+{\mathfrak s} y={\mathfrak s} (x+y)\\
A_7^{\rm Sk}:  & \ \   x\cdot 0=0 & A_8^{\rm Sk}: & \ \   x\cdot{\mathfrak s} y=x\cdot y + x
\end{align*}
Here ${\mathfrak p}$ abbreviates ${\mathfrak f}_{\exists
y(x={{\mathfrak s} }y)}$ and ${\mathfrak h}$ stands for ${\mathfrak f}_{\exists z
(x+z=y)}$.

\noindent For a fixed term $t$, put  $\Sigma_t$ be the following set of terms:
\newline\centerline{$\Sigma_t=\{0, t, t+0, {\mathfrak h}(t,0), {\mathfrak p}{\mathfrak h}(t,0), {\mathfrak s}{\mathfrak p}{\mathfrak h}(t,0), t+{\mathfrak s}{\mathfrak p}{\mathfrak h}(t,0), {\mathfrak s}(t+{\mathfrak s}{\mathfrak p}{\mathfrak h}(t,0))\}$,} and suppose that $p$ is an ${\rm Q}-$evaluation on $\Sigma_t$. We show that $p\models t\not\leqslant 0\vee t=0$. Note that Skolemizing $\psi$ results in $\psi^{\rm Sk}=(x\not\leqslant 0\vee x=0)$.  If $p$ is such an evaluation and if $p\models t\leqslant 0$, then by $A_4$ we have $p\models t+{\mathfrak h}(t,0)=0$. Now, either $p\models{\mathfrak h}(t,0)=0$ or $p\not\models{\mathfrak h}(t,0)=0$. In the former case, we have $p\models t+0=t$ which by $A_5$ implies $p\models t=0$. In the latter case, by $A_3$ we get $p\models {\mathfrak h}(t,0)={\mathfrak s}{\mathfrak p}{\mathfrak h}(t,0)$, and then  $p\models 0=t+{\mathfrak h}(t,0)=t+{\mathfrak s}{\mathfrak p}{\mathfrak h}(t,0)={\mathfrak s}(t+{\mathfrak p}{\mathfrak h}(t,0))$ by $A_6$,  which is a contradiction with  $A_1$. Thus we showed that if $p\models t\leqslant 0$ then necessarily $p\models t=0$.

\noindent Now, for two fixed  terms $u,v$ define $\Gamma_{u,v}$ as

$\Gamma_{u,v}=\{0, u, v, {\mathfrak s} v,  {\mathfrak h}(u,{\mathfrak s} v),  {\mathfrak p}{\mathfrak h}(u,{\mathfrak s} v), {\mathfrak s}{\mathfrak p}{\mathfrak h}(u,{\mathfrak s} v),  u+{\mathfrak p}{\mathfrak h}(u,{\mathfrak s} v),  $

\hspace{3.25em} $u+{\mathfrak s}{\mathfrak p}{\mathfrak h}(u,{\mathfrak s} v),   {\mathfrak s} (u+{\mathfrak p}{\mathfrak h}(u,{\mathfrak s} v))\}$.

\noindent We show that any
${\rm Q}-$evaluation on $\Gamma_{u,v}$  must satisfy $u\not\leqslant{\mathfrak s} v\vee u={\mathfrak s} v\vee u\leq v$. Note that the Skolemized form of $\varphi$ is $\varphi^{\rm Sk}=(x\not\leqslant{\mathfrak s} y\vee x={\mathfrak s} y\vee x\leqslant y)$. Suppose $p$ is an ${\rm Q}-$evaluation on $\Gamma_{u,v}$.
Then
either $p\models {\mathfrak h}(u,{\mathfrak s} v)=0$ or $p\models
{\mathfrak h}(u,{\mathfrak s} v)\not=0$. In the former case, by $A_4$,
we have $p\models u\not\leqslant{\mathfrak s} v\vee u+0={\mathfrak s} v$, and
then by $A_5$, $p\models u\not\leqslant{\mathfrak s} v\vee u={\mathfrak s} v$.
And in the latter case $p\models {\mathfrak h}(u,{\mathfrak s} v)={\mathfrak s} {\mathfrak p}{\mathfrak h}(u,{\mathfrak s} v)$ by $A_3$, also by $A_4$ we have $p\models u\not\leqslant{\mathfrak s} v\vee u+{\mathfrak s}{\mathfrak p}{\mathfrak h}(u,{\mathfrak s} v)={\mathfrak s} v$. On
the other hand from $A_5$ we get $p\models u+{\mathfrak s}{\mathfrak p}{\mathfrak h}(u,{\mathfrak s} v)={\mathfrak s} (u+{\mathfrak p}{\mathfrak h}(u,{\mathfrak s} v))$. Whence we get $p\models u\not\leqslant{\mathfrak s} v\vee{\mathfrak s}(u+{\mathfrak p}{\mathfrak h}(u,{\mathfrak s} v))={\mathfrak s} v$,
then by $A_2$,  $p\models u\not\leqslant{\mathfrak s} v\vee u+{\mathfrak p}{\mathfrak h}(u,{\mathfrak s} v)=v$, which by $A_4$ implies $p\models
u\not\leqslant{\mathfrak s} v\vee u\leqslant v$. Hence, in both cases we showed
$p\models u\not\leqslant{\mathfrak s} v\vee u={\mathfrak s} v\vee u\leqslant v$.
Finally, let us note that one could present a Herbrand proof of
${\rm Q}\vdash\psi$ and ${\rm Q}\vdash\varphi$ very similarly. \hfill$\subset\!\!\!\!\supset$ }\end{example}
\begin{example}\label{exampleB}{\rm
In the language of Example \ref{exampleA}, $\langle0,{\mathfrak s},+,\cdot,\leqslant\rangle$, let $\textrm{ind}_\psi$ be the following induction scheme for the formula $\psi(x)$:
\newline\centerline{
 $\psi(0)\wedge\forall x\big(\psi(x)\rightarrow\psi({\mathfrak s} x)\big)\rightarrow \forall x\psi(x)$.}

\noindent Assume for the moment that $\psi$ is an atomic formula. Then the Skolemization  of  $\textrm{ind}_\psi$ results in
$\textrm{ind}_\psi^{\rm Sk}: \ \neg\psi(0)\vee \Big(\psi({\mathfrak c})\wedge\neg\psi({\mathfrak s}{\mathfrak c})\Big)\vee \psi(x)$, where ${\mathfrak c}$ is the Skolem constant symbol ${\mathfrak f}_{\exists x\big(\psi(x)\wedge\neg\psi({\mathfrak s} x)\big)}$.
 Then any $\textrm{ind}_\psi-$evaluation $p$ on the set of terms $\{0,{\mathfrak c},{\mathfrak s}{\mathfrak c},t\}$ must satisfy one of the following:

 either (1) $p\not\models\psi(0)$ or (2) $p\models\psi({\mathfrak c})\wedge\neg\psi({\mathfrak s}{\mathfrak c})$ or (3) $p\models\psi(t)$.

\noindent Now take $\psi(x)$ to be the existential formula $\exists y\varphi(x,y)$ in which $\varphi$ is an atomic formula. Then the Skolemized form of $\textrm{ind}_\psi$ will be as
\newline\centerline{$\textrm{ind}_\psi^{\rm Sk}: \ \neg\varphi(0,u)\vee \Big(\varphi({\mathfrak c},{\mathfrak q}{\mathfrak c})\wedge\neg\varphi({\mathfrak s}{\mathfrak c},v)\Big)\vee \varphi\big(x,{\mathfrak q}(x)\big)$,}
 where ${\mathfrak q}$ is the Skolem function symbol for the formula ${\exists y\varphi(x,y)}$, and ${\mathfrak c}$ is the Skolem constant symbol for the sentence  ${\exists x\big(\exists w\varphi(x,w)\wedge\forall v\neg\varphi({\mathfrak s} x,v)\big)}$. The variables $u$, $v$ and $x$ are free.

\noindent We will need the case of $\varphi(x,y)= \big(y\leqslant x\cdot x\wedge y=x\cdot x\big)$ in the proof of Theorem~\ref{ij2} below. In this case the Skolemized form of $\textrm{ind}_\psi$ is
\begin{align*}
(u\not\leqslant 0^2\vee u\not=0^2) & \bigvee  \\
\Big(\big({\mathfrak q}{\mathfrak c}\leqslant{\mathfrak c}^2
\wedge{\mathfrak q}{\mathfrak c}={\mathfrak c}^2\big)\wedge\big(v\not\leqslant({\mathfrak s}{\mathfrak c})^2
\vee v\not=({\mathfrak s}{\mathfrak c})^2\big)\Big) & \bigvee  \\
\big({\mathfrak q}(x)\leqslant x^2\wedge {\mathfrak q}(x)=x^2\big). & \\
\end{align*}
The notation $\varrho^2$ is a shorthand for $\varrho\cdot\varrho$.
 Define the set of terms $\Upsilon$ by \newline\centerline{$\Upsilon=\{0,  0+0,  0^2, {\mathfrak c}, {\mathfrak c}^2, {\mathfrak c}^2+0, {\mathfrak s}{\mathfrak c}, {\mathfrak q}{\mathfrak c},  ({\mathfrak s}{\mathfrak c})^2,  ({\mathfrak s}{\mathfrak c})^2+0\}$} and suppose $p$ is an $(Q+\textrm{ind}_\psi)-$evaluation on the set of terms $\Upsilon\cup\{t,t^2,{\mathfrak q}(t)\}$. Then $p$ must satisfy the following Skolem instance $(\eth)$ of $\textrm{ind}_\psi$ which is available in the set $\Upsilon\cup\{t,t^2,{\mathfrak q}(t)\}$:
 \begin{align*}
(0\not\leqslant 0^2\vee 0\not=0^2) & \bigvee  \\
\Big(\big({\mathfrak q}{\mathfrak c}\leqslant{\mathfrak c}^2\wedge{\mathfrak q}{\mathfrak c}={\mathfrak c}^2\big)\wedge\big(({\mathfrak s}{\mathfrak c})^2\not\leqslant({\mathfrak s}{\mathfrak c})^2
\vee ({\mathfrak s}{\mathfrak c})^2\not=({\mathfrak s}{\mathfrak c})^2\big)\Big) & \bigvee  \\
\big({\mathfrak q}(t)\leqslant t^2\wedge {\mathfrak q}(t)=t^2\big). & \\
\end{align*}
Now since $p\models 0\cdot 0=0+0=0$ then, by ${\rm Q}$'s axioms, $p\models 0\leqslant 0^2\wedge 0=0^2$, and so $p$ cannot satisfy the first disjunct of $(\eth)$. Similarly, since $p\models ({\mathfrak s}{\mathfrak c})^2+0=({\mathfrak s}{\mathfrak c})^2$ then $p\models ({\mathfrak s}{\mathfrak c})^2\leqslant ({\mathfrak s}{\mathfrak c})^2$, thus $p$ cannot satisfy the second disjunct of $(\eth)$ either, because $p\models ({\mathfrak s}{\mathfrak c})^2=({\mathfrak s}{\mathfrak c})^2$. Whence,  $p$ must satisfy the third disjucnt of $(\eth)$, then necessarily $p\models {\mathfrak q}(t)=t^2$ must hold.
\hfill $\subset\!\!\!\!\supset$
}\end{example}
In Example \ref{exampleA}  we used the axioms of Robinson's Arithmetic ${\rm Q}$ to derive two sentences that will be needed later (see Lemma \ref{lemm1}). In Example \ref{exampleB} we used an axiom of ${\rm I\Delta_0}$ to derive the existence of an squaring Skolem function symbol (see the proof of Theorem \ref{ij2})
\begin{remark}\label{remark1}{\rm
The arguments of the above two examples can be generalized as follows: if $T\vdash\forall \overline{x}\theta(\overline{x})$ where $\theta$ is an open (quantifier-less) RNNF formula, then $(\neg\forall \overline{x}\theta(\overline{x}))^{\rm Sk}=\neg\theta(\overline{c})$ in which $\overline{c}$ is a sequence of Skolem constant symbols. There exists a set of terms $\Gamma$ (constructed from the Skolem function and constant symbols of $T$ with $\overline{c}$) such that there exists no $(T+\neg\forall\overline{x}\theta(\overline{x}))-$evluation on $\Gamma$. So, for any sequnece of terms $\overline{t}$, if $\Gamma(\overline{t})$ is the set of terms which result from the terms of $\Gamma$ by substituting $\overline{c}$ with $\overline{t}$, then any $T-$evaluation on $\Gamma(\overline{t})$ must satisfy the formula $\theta(\overline{t})$.
\hfill $\subset\!\!\!\!\supset$}
\end{remark}


\subsection{Arithmetization}
Fix ${\mathcal L}_A$ to be our language of arithmetic; one can set
${\mathcal L}_A=\langle 0,1,+,\cdot,<\rangle$ as e.g. in \cite{Kra95} or
${\mathcal L}_A=\langle 0,{\mathfrak s},+,\cdot,\leqslant\rangle$ as e.g. in
\cite{HP98}.  Later it will be clear that choosing this fixed language is not of much importance.

Peano's arithmetic ${\rm PA}$ is the first-order theory that extends ${\rm Q}$  (see Example \ref{exampleA}) by the following induction schema   for any arithmetical formula $\varphi(x)$:   \  $\varphi(0)\,\wedge\,\forall x\big(\varphi(x)\rightarrow\varphi(x+1)\big)\rightarrow\forall
x\varphi(x).$ Fragments of ${\rm PA}$ are extensions of ${\rm Q}$ with the induction schema restricted to a class of formulas. A formula is called bounded if its every quantifier is bounded, i.e., is either of the
form $\forall x\!\leqslant\! t(\ldots)$ or $\exists\, x\!\leqslant\! t(\ldots)$ where $t$ is a term; they are read as $\forall x (x\!\leqslant\! t\rightarrow\ldots)$ and $\exists x (x\!\leqslant\! t \wedge\ldots)$ respectively. It is easy to see that bounded formulas are decidable. The theory ${\rm I\Delta_0}$, also called bounded arithmetic, is axiomatized by ${\rm Q}$ plus the induction schema for bounded formulas.
The exponentiation function $\exp$ is defined by $\exp(x)=2^x$; the formula ${\rm Exp}$ expresses its totality: ($\forall x\exists y[y\!=\!\exp(x)]$). The converse of $\exp$ is denoted by $\log$ which is formally defined as $\log x={\rm min}\{y\mid x\leqslant\exp(y)\}$; and the cut $\ell\!{\it og}$ consists of the logarithms of all elements: $\ell\!{\it og}=\{x\mid \exists y [\exp(x)=y]\}$. The superscripts above the function symbols indicate the iteration of the functions: $\exp^2(x)=\exp(\exp(x))$, $\log^2 x=\log\log x$; similarly the cut $\ell\!{\it og}^n$ is $\{x\mid \exists y [\exp^n(x)=y]\}$. Let us recall that ${\rm Exp}$ is not provable in ${\rm I\Delta_0}$; and
sub-theories of ${\rm  I\Delta_0+Exp}$ are called weak arithmetics. Between ${\rm  I\Delta_0}$ and ${\rm I\Delta_0+Exp}$ a hierarchy of theories is considered in the literature, which has close connections with computational complexity. Define the function $\omega_m$ to be $\omega_m(x)=\exp^m\big((\log^m x)\cdot(\log^m x)\big)$. It is customary to define this function by induction:
 ${\rm \omega_0}(x)=x^2$ and ${\rm \omega_{n+1}}(x)=\exp({\rm \omega_n}(\log x))$. Let
 ${\rm \Omega_m}$ express the totality of ${\rm \omega_m}$ (i.e., ${\rm \Omega_m}\equiv \forall x\exists y [y={\rm \omega_m}(x)]$).

 By G\"odel's coding method, we are now rest assured that the concepts
introduced in the pervious section all can be formalized (and
arithmetized) in the language of arithmetic. But we need just a bit
more; and that is an ``effective" coding, suitable for bounded
arithmetic. The one we adopt here is taken from Chapter~V of \cite{HP98}.  For
convenience, and shortening the computations, we introduce the ${\mathcal P}$ notation.

\begin{definition}{\rm
We say $x$ is of ${\mathcal P}(y)$, when  the code of $x$ is bounded  above by a polynomial of $y$; and we write this as $\ulcorner x\urcorner\leqslant{\mathcal P}(y)$, meaning that for some $n$ the inequality $\ulcorner x\urcorner\leqslant y^n+n$ holds.}\hfill $\subset\!\!\!\!\supset$
\end{definition}

Let us note that $X\leqslant{\mathcal P}(Y)$ is equivalent to the old (more familiar) $O-$notation $``\log  X\!
\in\!{\mathcal O}(\log Y)"$. Here we collect some very basic facts about this fixed efficient coding  that will be needed later.

\begin{remark}{\rm Let $A$ be a set or a sequence of terms, and let $|A|$ denote the cardinality of $A$, when $A$ is a set, and the same $|A|$ denote the length of $A$, when $A$ is a sequence. Then
\begin{itemize}
\item $\ulcorner\langle\alpha\rangle\urcorner \leqslant 9(\ulcorner\alpha\urcorner+1)^2$ (Lemma 3.7.2 page 297 of \cite{HP98});
\item $\ulcorner A\!\frown\! B\urcorner \ \left(\ulcorner A\cup B\urcorner\right)
\leqslant 64\cdot(\ulcorner A\urcorner\cdot \ulcorner B\urcorner)$ (Proposition 3.29 page 311  of \cite{HP98});
\item $\left(|A|\right)\leqslant (\log \ \ulcorner A\urcorner)$ (Definition 3.27 
and Section (e) pages 304--310  of \cite{HP98});
\end{itemize}
where $\ulcorner A\!\frown\! B\urcorner$ is the concatenation of (the sequences) $A$ and $B$. \hfill $\subset\!\!\!\!\supset$
}\end{remark}

If we let ${\mathcal L}_A^{\rm Sk}$ to be the closure of ${\mathcal L}_A$ under Skolem function and constant symbols, i.e., let ${\mathcal L}_A^{\rm Sk}$  be the smallest set that contains ${\mathcal L}_A$ and for any ${\mathcal L}_A^{\rm Sk}-$formula $\exists x \phi(x)$ we have ${\mathfrak f}_{\exists x \phi(x)}\!\in\!{\mathcal L}_A^{\rm Sk}$, then
  this new countable language can also be re-coded, and this recoding can be generalized to ${\mathcal L}_A^{\rm Sk}-$terms and ${\mathcal L}_A^{\rm Sk}-$formulas.
We wish to compute an upper bound for the codes of evaluations on a
set of terms $\Lambda$. For a given $\Lambda$, all the atomic
formulas, in the language ${\mathcal L}_A$, constructed from
terms of $\Lambda$ are either of the form $t=s$ or of the form
$t\leqslant s$ for some $t,s\!\in\!\Lambda$. And every member of an
evaluation $p$ on $\Lambda$ is an ordered pair like $\langle
t=s,i\rangle$ or $\langle t\leqslant s,i\rangle$ for some $t,s\!\in\!\Lambda$
and $i\!\in\!\{0,1\}$. Thus the code of any member of $p$ is a constant multiple of
$(\ulcorner t\urcorner\cdot \ulcorner s\urcorner)^2$, and so the code of $p$ is bounded above by
${\mathcal P}(\prod_{t,s\in\Lambda}\ulcorner t\urcorner\cdot \ulcorner s\urcorner)$.

\begin{lemma}\label{codeomega1}
For a set of terms $\Lambda$ and evaluation $p$ on it,
{$\ulcorner p\urcorner \leqslant{\mathcal P}\left(\omega_1(\ulcorner\Lambda\urcorner)\right)$.}
\end{lemma}
\begin{proof}
It suffices, by the above remark and what was said afterward, to note that
 $\prod_{t,s\in\Lambda}\ulcorner t\urcorner\cdot
\ulcorner s\urcorner=\prod_{t\in\Lambda}(\ulcorner t\urcorner)^{2|\Lambda|}=
(\prod_{t\in\Lambda}\ulcorner t\urcorner)^{2|\Lambda|}\leqslant{\mathcal P}(\ulcorner\Lambda\urcorner)
^{2\log\ulcorner\Lambda\urcorner}\leqslant{\mathcal P}(\ulcorner\Lambda\urcorner^{\log\ulcorner\Lambda\urcorner})$
and that
 $\ulcorner\Lambda\urcorner^{\log\ulcorner\Lambda\urcorner}\leqslant
\omega_1(\ulcorner\Lambda\urcorner)$.
\end{proof}

Let us have another look at the above lemma, which is of great importance. For a set of terms $\Lambda$, there are $|\Lambda|$ terms in it (the cardinality of $\Lambda$). So, there are $2|\Lambda|^2$ atomic formulas constructed from the terms of $\Lambda$ (atomic formulas of the form $t=s$ or $t\leqslant s$ for $t,s\!\in\!\Lambda$). And thus, there are $\exp(2|\Lambda|^2)$ different evaluations on the set $\Lambda$. And finally note that by $|\Lambda|\leqslant (\log\ulcorner\Lambda\urcorner)$ we get  $\exp(2|\Lambda|^2)\leqslant {\mathcal P}\left(\exp((\log\ulcorner\Lambda\urcorner)^2)\right)
\leqslant{\mathcal P}\left(\omega_1(\ulcorner\Lambda\urcorner)\right)$. So, in the presence of $\omega_1(\ulcorner\Lambda\urcorner)$ we have all the evaluations on $\Lambda$ in our disposal.


All these concepts can be expressed in the language of arithmetic
${\mathcal L}_A$ by appropriate formulas. And ``Herbrand Consistency of the
 theory $T$" can be arithmetized as ``for
every set of terms there exists an $T-$evaluation on it".
Let ${\rm HCon}(T)$ denote the ${\mathcal L}_A-$formula ``$T$ is
Herbrand consistent".


\section{Herbrand Models}
For a theory $T$, when $\Lambda$ is the set of all terms
(constructed from the function symbols of the language of $T$ and
also the Skolem function symbols of the formulas of $T$) any
$T-$evaluation on $\Lambda$ induces a model of $T$, which is called
a {\em Herbrand model}.
Here we use this notion  for building up a definable inner model, which will constitute the hear of the proof of our main result for ${\rm I\Delta_0+\Omega_1}$.

\subsection{Arithmetically Definable Herbrand Models}
In the sequel, we arithmetize Herbrand models.

\begin{definition}{\rm
Let ${\mathcal L}$ be a language  and $\Lambda$ be a
 set of (ground) terms (constructed by the Skolem constant and function symbols of ${\mathcal L}$).

\noindent Put $\Lambda^{\langle 0\rangle}=\Lambda$, and define inductively

$\Lambda^{\langle k+1\rangle}=\Lambda^{\langle k\rangle}\cup
\{f(t_1,\ldots,t_m)\mid f\!\in\!{\mathcal L}
\;\&\;t_1,\ldots,t_m\!\in\!\Lambda^{\langle k\rangle}\}$

\hspace{5.9em}  $\cup  \,\{{\mathfrak f}_{\exists
x\psi(x)}(t_1,\ldots,t_m)\mid \ulcorner\psi\urcorner\leqslant k \;\&\;
t_1,\ldots,t_m\!\in\!\Lambda^{\langle k\rangle}\}$.

\noindent Let $\Lambda^{\langle \infty\rangle}$ denote the union
$\bigcup_{k\in\mathbb{N}}\Lambda^{\langle k\rangle}$.

\noindent Suppose $p$ is an evaluation on $\Lambda^{\langle \infty\rangle}$.
Define ${\textswab M}(\Lambda,p)=\{t/p\mid t\!\in\!\Lambda^{\langle
\infty\rangle}\}$ and put the ${\mathcal L}-$structure on it by
\begin{itemize}
\item $f^{{\textswab M}(\Lambda,p)}(t_1/p,\ldots,t_m/p)=f(t_1,\ldots,t_m)/p$, and
\item $R^{{\textswab M}(\Lambda,p)}=\{(t_1/p,\ldots,t_m/p)\mid p\models
R(t_1,\ldots,t_m)\}$;
\end{itemize}
for $f,R\!\in\!{\mathcal L}$ and
$t_1,\ldots,t_m\!\in\!\Lambda^{\langle\infty\rangle}$. \hfill $\subset\!\!\!\!\supset$
}\end{definition}

\begin{lemma}\label{lm:mlambdap}{\rm
The definition of ${\mathcal L}-$structure on ${\textswab M}(\Lambda,p)$ is
well-defined, and when $p$ is an $T-$evaluation on
$\Lambda^{\langle\infty\rangle}$, for an ${\mathcal L}-$theory $T$, then
${\textswab M}(\Lambda,p)\models T$. }\end{lemma}

\begin{proof} That the definitions of  $f^{{\textswab M}(\Lambda,p)}$ and $R^{{\textswab M}(\Lambda,p)}$ are well-defined follows directly from  the definition of an evaluation (Definition \ref{defeval}). By the definition of $\Lambda^{\langle\infty\rangle}$ the structure ${\textswab M}(\Lambda,p)$ is closed under all the Skolem functions of ${\mathcal L}$, and moreover it satisfies an atomic (or negated atomic) formula
$A(t_1/p,\ldots,t_m/p)$ if and only if $p\models A(t_1,\ldots,t_m)$.
Then it can be shown, by induction on the complexity of formulas,
that for every RNNF formula $\psi$, we have ${\textswab M}(\Lambda,p)\models\psi$ whenever $p$ satisfies all the available
Skolem instances of $\psi$ in $\Lambda^{\langle\infty\rangle}$.
\end{proof}

We need an  upper bound on the size (cardinal) and the code of $\Lambda^{\langle j\rangle}$.

\begin{lemma}\label{lambdacode}
The following inequalities hold when $\ulcorner\Lambda\urcorner$ and
$|\Lambda|$ are sufficiently larger than $n$:

{\rm (1)}  \ $|\Lambda^{\langle n\rangle}|\leqslant{\mathcal P}\left( |\Lambda|^{n!}\right)$, and

{\rm (2)} \  $\ulcorner\Lambda^{\langle n\rangle}\urcorner\leqslant
{\mathcal P}\Big(\big(\ulcorner\Lambda\urcorner\big)^{|\Lambda|^{(n+1)!}}\Big)$.
\end{lemma}

\begin{proof}
Denote $\ulcorner\Lambda^{\langle k\rangle}\urcorner$ by $\lambda_k$ (thus $\ulcorner\Lambda\urcorner=\lambda_0=\lambda$) and $|\Lambda^{\langle k\rangle}|$ by $\sigma_k$ (and thus $|\Lambda|=\sigma_0=\sigma$). We first note that $\sigma_{k+1}\leqslant \sigma_k+M\sigma_k^M+k\sigma_k^k$ for a fixed $M$. Thus $\sigma_{k+1}\leqslant{\mathcal P}(\sigma_k^{k+1})$, and then, by an inductive argument, we have $\sigma_n\leqslant{\mathcal P}(\sigma^{n!})$. For the second statement, we first compute an upper bound for the code of the Cartesian power $A^m$ for a set $A$. By an argument similar to that of the proof of Lemma~\ref{codeomega1}, we have $\ulcorner A^{k+1}\urcorner\leqslant {\mathcal P}\big(\prod_{t\in A^k  \&  s\in A}\ulcorner t\urcorner \cdot \ulcorner s\urcorner\big)\leqslant {\mathcal P}\big( \ulcorner A^k\urcorner^{|A|}\cdot\ulcorner A\urcorner^{|A|^k}\big)$, and thus $\ulcorner A^m\urcorner \leqslant {\mathcal P}\big(\ulcorner A\urcorner ^{|A|^m}\big)$ can be shown by induction on $m$. Now we have $\lambda_{k+1}\leqslant{\mathcal P}\big(\ulcorner \Lambda^{\langle k\rangle}\urcorner\cdot\ulcorner (\Lambda^{\langle k\rangle})^M\urcorner\cdot\ulcorner (\Lambda^{\langle k\rangle})^k\urcorner\big)$ for a fixed $M$. So, $\lambda_{k+1}\leqslant {\mathcal P}\big(\lambda_k^{{\sigma_k}^{k}}\big)$ and finally our desired conclusion $\lambda_m\leqslant{\mathcal P}\big(\lambda^{\sigma^{(m+1)!}}\big)$ follows by induction.
\end{proof}

Stating the above fact as a lemma, despite of the fact that it  is indeed a crucial tool  for our arguments,
let us state the following corollary of it as a theorem, and later on we will use the theorem and will leave the lemma right here.

\begin{theorem}\label{log4}
If for a set of terms $\Lambda$ with non-standard
$\ulcorner\Lambda\urcorner$ the value
$\omega_2(\ulcorner\Lambda\urcorner)$ exists, then for a non-standard
$j$ the value $\ulcorner\Lambda^{\langle j\rangle}\urcorner$ will
exist.
\end{theorem}

\begin{proof}
There must exist a non-standard $j$ such that $j\leqslant\log^4(
\ulcorner\Lambda\urcorner)$. Thus $2(j+1)!\leqslant
2^{2^j}\leqslant\log^2\ulcorner\Lambda\urcorner$. Now, by Lemma~\ref{lambdacode} we can write
$\ulcorner\Lambda^{\langle
j\rangle}\urcorner\leqslant{\mathcal P}\left((\ulcorner\Lambda\urcorner)^{|\Lambda|^{(j+1)!}}\right)\leqslant{\mathcal P}
\left((2^{2\log\ulcorner\Lambda\urcorner})^{(\log\ulcorner\Lambda\urcorner)^{
(j+1)!}}\right)\leqslant$

\noindent ${\mathcal P}\left(\exp((\log\ulcorner\Lambda\urcorner)^{2(j+1)!})\right)
\leqslant{\mathcal P}\left(\exp(\omega_1(\log\ulcorner\Lambda\urcorner))\right)$,
 or in other words
$\ulcorner\Lambda^{\langle
j\rangle}\urcorner\leqslant{\mathcal P}\left(\omega_2(\ulcorner\Lambda\urcorner)\right)$.
\end{proof}

The reason that Theorem \ref{log4} is stated for non-standard
$\Lambda$ is that the set $\Lambda^{\langle\infty\rangle}$, needed
for constructing the model ${\textswab M}(\Lambda,p)$,  is not
definable in ${\mathcal L}_A$. But the existence of the definable
$\Lambda^{\langle j\rangle}$ for a non-standard $j$ can guarantee
the existence of $\Lambda^{\langle\infty\rangle}$ and thus of
${\textswab M}(\Lambda,p)$. This non-standard $j$ exists for
non-standard $\ulcorner\Lambda\urcorner$.

\subsection{The Main Theorem for ${\rm I\Delta_0+\Omega_1}$}
Two interesting theorems were proved by Z. Adamowicz in \cite{Ada02} about Herbrand Consistency of the theories  ${\rm I\Delta_0+\Omega_m}$ for $m\geqslant 2$:

\begin{theorem}[Z. Adamowicz  \cite{Ada02}]\label{th1}
For a bounded formula $\theta(\overline{x})$ and $m\geqslant 2$, if the theory
{$({\rm I\Delta_0+\Omega_m}) \,+\, \exists \overline{x}\!\in\!\ell\!{\it og}^{m+1}\theta(\overline{x}) \,+\, {\rm HCon}_{\ell\!{\it og}^{m-2}}({\rm I\Delta_0+\Omega_m})$} is consistent, then so is the theory {$({\rm I\Delta_0+\Omega_m})  \,+\, \exists \overline{x}\!\in\!\ell\!{\it og}^{m+2}\theta(\overline{x}),$} where ${\rm HCon}_{\ell\!{\it og}^{m-2}}$ is the relativization of ${\rm HCon}$ to the cut $\ell\!{\it og}^{m-2}$.\hfill $\subset\!\!\!\!\supset$
\end{theorem}

\begin{theorem}[Z. Adamowicz  \cite{Ada02}]\label{th2}
For any natural $m,n\geqslant 0$ there exists a bounded formula $\eta(\overline{x})$ such that {$({\rm I\Delta_0+\Omega_m}) \,+\, \exists \overline{x}\!\in\!\ell\!{\it og}^{n}\eta(\overline{x})$} is consistent, but  {$({\rm I\Delta_0+\Omega_m}) \,+\, \exists \overline{x}\!\in\!\ell\!{\it og}^{n+1}\eta(\overline{x})$} is not consistent.\hfill $\subset\!\!\!\!\supset$
\end{theorem}

\noindent These two theorems (by putting $n=m+1$ for $m\geqslant 2$) imply together that for any $m\geqslant 2:$
 \newline\centerline{
 ${\rm I\Delta_0+\Omega_m}\not\vdash{\rm HCon}_{\ell\!{\it og}^{m-2}}({\rm I\Delta_0+\Omega_m}).$}

Here we extend Theorem \ref{th1} for ${\rm I\Delta_0+\Omega_1}$, namely we show that

\begin{theorem}\label{O1H}
For any bounded formula $\theta(x)$, the consistency of the theory
{$({\rm I\Delta_0+\Omega_1}) \,+\, \exists
x\!\!\in\!\!\ell\!{\it og}^2\theta(x) \,+\, {\rm HCon}({\rm I\Delta_0+\Omega_1})$} implies the consistency of
 the theory
 {$({\rm I\Delta_0+\Omega_1}) \,+\, \exists
x\!\!\in\!\!\ell\!{\it og}^3\theta(x)$.}
\end{theorem}

The rest of this section is devoted to proving this theorem. Let us note that Theorem \ref{th2} holds already for ${\rm I\Delta_0+\Omega_1}$, and below we reiterate the part that we need here:

\begin{theorem}[Z. Adamowicz  \cite{Ada02}]
There exists a bounded formula $\eta(\overline{x})$ such that the arithmetical theory {$({\rm I\Delta_0+\Omega_1}) \,+\, \exists \overline{x}\!\in\!\ell\!{\it og}^{2}\eta(\overline{x})$} is consistent, but  {$({\rm I\Delta_0+\Omega_1}) \,+\, \exists \overline{x}\!\in\!\ell\!{\it og}^{3}\eta(\overline{x})$} is not consistent.\hfill $\subset\!\!\!\!\supset$
\end{theorem}

Having proved the main theorem (\ref{O1H}), we can immediately infer that $${\rm I\Delta_0+\Omega_1}\not\vdash{\rm HCon}({\rm I\Delta_0+\Omega_1}).$$

As the proof of Theorem \ref{O1H} is long, we will break it into a few lemmas.
First we note that  $\alpha\!\in\!\ell\!{\it og}^3$ if and only if there exists a
sequence $\langle w_0,w_1,\cdots,w_\alpha\rangle$ of length $(\alpha+1)$ such
that $w_0=\exp^3(0)=2^2$,  and
for any $j<\alpha$, $w_{j+1}=\omega_1(w_j)$. Noting that
$\omega_1(\exp^3(j))=\exp^3(j+1)$ one can then see that
$w_\alpha=\exp^3(\alpha)$, and so $\alpha\!\in\!\ell\!{\it og}^3$. This can be formalized in
${\rm I\Delta_0+\Omega_1}$ by an arithmetical formula. Note that the code of the above
sequence is bounded by ${\mathcal P}(\prod_{j=0}^{j=\alpha}w_j)\leqslant
{\mathcal P}\big(\exp(\sum_{j=0}^{j=\alpha}
\exp^2(j))\big)\leqslant{\mathcal P}\left(\exp^3(\alpha+1)\right)\leqslant
{\mathcal P}\left(\omega_1(\exp^3(\alpha))\right)$. So, in the presence of
$\Omega_1$, the existence of $\exp^3(\alpha)$ guarantees the existence of
the above sequence of $w_j$'s.

For proving Theorem \ref{O1H} let us assume that we have a model
$${\mathcal M}\models
({\rm I\Delta_0+\Omega_1})+\big(\alpha\!\in\!\ell\!{\it og}^2\wedge\theta(\alpha)\big)+{\rm HCon}({\rm I\Delta_0+\Omega_1}),$$ for some  bounded formula $\theta(x)$ and some non-standard
$\alpha\!\in\!{\mathcal M}$, and then we construct a model  $${\mathcal N}\models ({\rm I\Delta_0+\Omega_1})+\exists
x\!\!\in\!\!\ell\!{\it og}^3\theta(x).$$

If our language of arithmetic ${\mathcal L}_A$ contains the successor function ${\mathfrak s}$, then define the terms $\underline{j}$'s by induction: $\underline{0}=0$, and  $\underline{j+1}={\mathfrak s}(\underline{j})$. If ${\mathcal L}_A$ does not contain ${\mathfrak s}$, then it should have the constant $1$, and in this case we can put $\underline{j+1}=\underline{j}+1$. The term $\underline{j}$ represents the (standard or non-standard) number $j$. For
the sake of simplicity, assume ${\mathfrak w}$ denotes the Skolem function
symbol ${\mathfrak f}_{\exists y\left(y=\omega_1(x)\right)}$. Put ${\sf
w}_0=\underline{4}$ and inductively ${\sf w}_{j+1}={\mathfrak w}({\sf w}_j)$. Then
${\sf w}_k$, in the theory ${\rm I\Delta_0+\Omega_1}$, is the term which represents
$\exp^3(k)$. Finally, put
$\Lambda=\{\underline{0},\ldots,\underline{\omega_1(\alpha)},{\sf w}_0,\ldots,{\sf
w}_\alpha\}=\{\underline{j}\mid j\leqslant \omega_1(\alpha)\}\cup\{{\sf w}_j\mid j\leqslant \alpha\}$. We can now estimate an  upper bound for the code of
$\Lambda$:
$\ulcorner\Lambda\urcorner\leqslant{\mathcal P}\left(\prod_{j=1}^{j=\omega_1(\alpha)}2^j\right)
\leqslant{\mathcal P}\left(\exp(\omega_1(\alpha)^2)\right)$.

\noindent So $\Lambda$ has a code in ${\mathcal M}$ (since ${\mathcal M}\models
\alpha\!\in\!\ell\!{\it og}^2$), and moreover $\omega_2(\ulcorner\Lambda\urcorner)$
exists in ${\mathcal M}$, because $\omega_2(\ulcorner\Lambda\urcorner)\leqslant{\mathcal P}\left(\omega_2(\exp(\omega_1(\alpha)^2))\right)
\leqslant{\mathcal P}\left(\exp(\omega_1(\omega_1(\alpha)^2))\right)
\leqslant{\mathcal P}\left(\exp^2\left(4(\log\alpha)^4\right)\right)\leqslant{\mathcal P}\left(\exp^2(\alpha)\right)$.

\noindent Thus by Theorem \ref{log4} there exists a non-standard $j$
such that $\Lambda^{\langle j\rangle}$ has a code in ${\mathcal M}$. Since by the assumption above  we have
${\mathcal M}\models{\rm HCon}({\rm I\Delta_0+\Omega_1})$, then there exists an $({\rm I\Delta_0+\Omega_1})-$evaluation $p$
on $\Lambda^{\langle j\rangle}$ (in ${\mathcal M}$). Now, by what was said
after the proof of Theorem \ref{log4} one can construct the model
${\textswab M}(\Lambda,p)={\mathcal N}$. By Lemma~\ref{lm:mlambdap} we have
${\mathcal N}\models({\rm I\Delta_0+\Omega_1})$, and also ${\mathcal N}\models\underline{\alpha}/p\!\in\!\ell\!{\it og}^3$ follows from the existence of ${\sf w}_j/p$'s. It remains (only) to show
that ${\textswab M}(\Lambda,p)\models\theta(\underline{\alpha}/p).$

For this purpose we prove the following lemmas where we assume that ${\mathcal M}$ is as
 above and there are some non-standard set of
terms and evaluation $\Lambda,p$ in ${\mathcal M}$ such that
$\Lambda\supseteq\{\underline{0},\ldots,\underline{\omega_1(\alpha)}\}$ for a non-standard $\alpha\!\in\!{\mathcal M}$, and $p$ is an
${\rm I\Delta_0}-$evaluation on $\Lambda^{\langle\infty\rangle}$.

\begin{lemma}\label{lemm1}
If ${\textswab M}(\Lambda,p)\models t/p\leqslant \underline{i}/p$ holds for a
term $t$ and $i\leqslant\omega_1(\alpha)$ in ${\mathcal M}$, then ${\textswab M}(\Lambda,p)\models
t/p=\underline{j}/p$  for some $j\leqslant i$.
\end{lemma}
\begin{proof}
By the assumption ${\mathcal M}\models ``p\models t\leqslant \underline{i}"$. We prove by induction on $i$ that there exists some $j\leqslant i$ in ${\mathcal M}$ such that ${\mathcal M}\models ``p\models t=\underline{j}"$.

$\bullet$ For $i=0$ by Example \ref{exampleA}  the assumption ${\mathcal M}\models ``p\models t\leqslant 0"$ implies ${\mathcal M}\models ``p\models t=0"$, noting that $p$ is an ${\rm Q}-$evaluation on $\Lambda^{\langle\infty\rangle}$, and thus all the needed Skolem terms are in $p$'s disposal.

$\bullet$ For $i+1$ we have ${\mathcal M}\models ``p\models t\leqslant \underline{i}\bigvee t={\mathfrak s}(\underline{i})"$ by Example~\ref{exampleA}  and the assumed satisfaction  ${\mathcal M}\models ``p\models t\leqslant{\mathfrak s}(\underline{i})"$. Then if ${\mathcal M}\models ``p\models t={\mathfrak s}(\underline{i})"$ we are done, and if ${\mathcal M}\models ``p\models t\leqslant \underline{i}"$ by the induction hypothesis there must exist some $j\leqslant i$ in ${\mathcal M}$ such that ${\mathcal M}\models ``p\models t=\underline{j}"$.
\end{proof}
\begin{remark}\label{remark2}{\rm
The proof of the above lemma does not depend on the axioms of ${\rm Q}$ (and ${\rm I\Delta_0}$). Indeed, in some axiomatization of ${\rm Q}$ in the literature, the sentences \newline\centerline{$\psi=\forall x(x\leqslant 0\rightarrow x=0)$ and $\varphi=\forall x\forall y(x\leqslant{\mathfrak s} y\rightarrow x={\mathfrak s} y\vee x\leqslant y)$ (see Example \ref{exampleA})} are accepted as axioms. In our axiomatization, the above sentences were derivable theorems. In some axiomatizations of ${\rm Q}$ our axiom $A_4$ is replaced with $A_4': \forall x,y (x\leqslant y\leftrightarrow\exists z [z+x=y])$; note the difference of $x+z$ in $A_4$ and $z+x$ in $A_4'$ (see e.g. \cite{HP98}). In this new axiomatization the sentence $\varphi$ is not derivable. However, since we have ${\rm I\Delta_0}\vdash\psi\wedge\varphi$, then by the argument of Remark \ref{remark1}, the above Lemma \ref{lemm1} can be proved by using the fact that $p$ is an ${\rm I\Delta_0}-$evaluation on $\Lambda^{\langle\infty\rangle}$.
\hfill$\subset\!\!\!\!\supset$
}\end{remark}
\begin{lemma}\label{lemm2}
For  any ${\mathcal L}_A-$term $t(x_1,\ldots,x_m)$ and $i_1,\ldots,i_m\leqslant
\omega_1(\alpha)$, if ${\mathcal M}\models x\leqslant t(i_1,\ldots,i_m)$ for some $x$, then for
an ${\mathcal L}_A-$term  $t'(x_1,\ldots,x_k)$  and some $j_1,\ldots,j_k\leqslant
\omega_1(\alpha)$ we have  ${\mathcal M}\models x=t'(j_1,\ldots,j_k)$.
\end{lemma}
\begin{proof} By induction on (the complexity of) the term $t$.

$\bullet$ For
$t=0$  and  $t=x_1$ the proof is straightforward.

$\bullet$ For $t={\mathfrak s} u$ the assumption ${\mathcal M}\models x\leqslant {\mathfrak s} u(i_1,\ldots,i_m)$ implies that either ${\mathcal M}\models x={\mathfrak s} u(i_1,\ldots,i_m)$ or ${\mathcal M}\models x\leqslant u(i_1,\ldots,i_m)$ is true, and then
the conclusion follows from the induction hypothesis.

$\bullet$ For $t=u+v$, and the assumption ${\mathcal M}\models x\leqslant
u(i_1,\ldots,i_m)+v(i_1,\ldots,i_m)$, we consider two cases. First
if ${\mathcal M}\models x\leqslant u(i_1,\ldots,i_m)$ then we are done by the
induction hypothesis. Second if ${\mathcal M}\models u(i_1,\ldots,i_m)\leqslant
x$ then there exists a $y$ such that ${\mathcal M}\models
x=u(i_1,\ldots,i_m)+y$ and moreover ${\mathcal M}\models y\leqslant
v(i_1,\ldots,i_m)$. Now, by the induction hypothesis there are a
term $t'(x_1,\ldots,x_k)$ and some elements $j_1,\ldots,j_k\leqslant \omega_1(\alpha)$ such that
${\mathcal M}\models y=t'(j_1,\ldots,j_k)$. Whence we  finally get the conclusion
${\mathcal M}\models x=u(i_1,\ldots,i_m)+t'(j_1,\ldots,j_k)$.

$\bullet$ For $t=u\cdot v$, by an argument similar to that of
the previous case, we can assume ${\mathcal M}\models u(i_1,\ldots,i_m)\leqslant
x\leqslant u(i_1,\ldots,i_m)\cdot v(i_1,\ldots,i_m)$. There are some $q,r$ such that ${\mathcal M}\models
x=u(i_1,\ldots,i_m)\cdot q+r$ and ${\mathcal M}\models r\leqslant
u(i_1,\ldots,i_m)$. We also have ${\mathcal M}\models q\leqslant
v(i_1,\ldots,i_m)$. By the induction hypothesis there are terms
$t',t''$ and $j_1,\ldots,j_k\leqslant \omega_1(\alpha)$ such that ${\mathcal M}\models
q=t'(j_1,\ldots,j_k) \bigwedge r=t''(j_1,\ldots,j_k)$. Thus we finally have ${\mathcal M}\models
x=u(i_1,\ldots,i_m)\cdot t'(j_1,\ldots,j_k)+t''(j_1,\ldots,j_k)$.
\end{proof}

\begin{lemma}\label{lemm3}
For $i,j,k\leqslant\omega_1(\alpha)$ in ${\mathcal M}$ we have

(1) if $i\leqslant j\leqslant\omega_1(\alpha)$ then ${\textswab M}(\Lambda,p)\models
\underline{i}/p\leqslant \underline{j}/p$ ;

(2) if $i+j\leqslant\omega_1(\alpha)$ then ${\textswab M}(\Lambda,p)\models
\underline{i}/p+\underline{j}/p=\underline{{i+j}}/p$ ;

(3) if $i\cdot j\leqslant\omega_1(\alpha)$ then ${\textswab M}(\Lambda,p)\models
\underline{i}/p\cdot\underline{j}/p=\underline{{i\cdot j}}/p$ .
\end{lemma}
\begin{proof}
We need to show for the  $i,j\leqslant\omega_1(\alpha)$ that

(1) if ${\mathcal M}\models i\leqslant j$ then ${\mathcal M}\models ``p\models \underline{i}\leqslant \underline{j}"$,

(2) if ${\mathcal M}\models i+j\leqslant\omega_1(\alpha)$ then ${\mathcal M}\models ``p\models \underline{i}+\underline{j}=\underline{{i+j}}"$, and

(3) if ${\mathcal M}\models i\cdot j\leqslant\omega_1(\alpha)$ then  ${\mathcal M}\models ``p\models \underline{i}\cdot\underline{j}=\underline{{i\cdot j}}"$.

\noindent First we note that the statement (2) above implies already (1), since if  we have ${\mathcal M}\models i\leqslant j$, then for some $k$ we should have ${\mathcal M}\models i+k=j$, and then by (2), ${\mathcal M}\models ``p\models \underline{i}+\underline{k}=\underline{j}"$ which implies (by $A_4$ of ${\rm Q}$ - see Example \ref{exampleA}) that ${\mathcal M}\models ``p\models \underline{i}\leqslant \underline{j}"$.  By induction on $j$, very similarly to the proof of Lemma \ref{lemm1},  one can prove the statements (2) and (3),  noting that the evaluation $p$ must satisfy the following axioms of ${\rm Q}$:
\begin{eqnarray*}
A_5:&   \forall x (x+0=x);  \ \ \ \ \ \ \   & A_6:    \forall x\forall y (x+{\mathfrak s} y={\mathfrak s} (x+y)); \\
A_7:&  \forall x (x\cdot 0=0);   \ \ \ \ \ \ \ & A_8:    \forall x\forall y (x\cdot{\mathfrak s} y=x\cdot y + x).
\end{eqnarray*}\end{proof}

\begin{corollary}\label{cor1}
Suppose for an  ${\mathcal L}_A-$term $t(x_1,\ldots,x_m)$ and some elements
$i_1,\ldots,i_m,i\leqslant\omega_1(\alpha)$, we have ${\mathcal M}\models
t(i_1,\ldots,i_m)=i$. Then we must also have ${\textswab M}(\Lambda,p)\models
t(\underline{i_1}/p,\ldots,\underline{i_m}/p)=\underline{i}/p$.
\end{corollary}
\begin{proof} By induction on $t$ using Lemma \ref{lemm3}.
\end{proof}

\begin{lemma}\label{lem4}
Suppose   $t(x_1,\ldots,x_m),t'(x_1,\ldots,x_m)$ are two
${\mathcal L}_A-$terms and $i_1,\ldots,i_m\leqslant \alpha^k$ are  elements of
${\mathcal M}$ for some standard number  $k\!\in\!\mathbb{N}$.
Then, if  ${\mathcal M}\models t(i_1,\ldots,i_m) = t'(i_1,\ldots,i_m)$
holds,
${\textswab M}(\Lambda,p)\models
t(\underline{{i_1}}/p,\ldots,\underline{{i_m}}/p) =
t'(\underline{{i_1}}/p,\ldots,\underline{{i_m}}/p)$ must hold too.
\end{lemma}
\begin{proof} By $i_1,\ldots,i_m\leqslant \alpha^k$ we have $t(i_1,\ldots,i_m)\leqslant\omega_1(\alpha)$. Put  $i$ be the common value {$i=t(i_1,\ldots,i_m) = t'(i_1,\ldots,i_m)$.}
Then By Corollary \ref{cor1} we have
{${\textswab M}(\Lambda,p)\models
t(\underline{{i_1}}/p,\ldots,\underline{{i_m}}/p) =\underline{i}/p=
t'(\underline{{i_1}}/p,\ldots,\underline{{i_m}}/p).$}
\end{proof}

\begin{lemma}\label{lem5}
Suppose   $t(x_1,\ldots,x_m),t'(x_1,\ldots,x_m)$ are two
${\mathcal L}_A-$terms and $i_1,\ldots,i_m\leqslant \alpha^k$ are  elements of
${\mathcal M}$ for some standard number  $k\!\in\!\mathbb{N}$.
If we have ${\mathcal M}\models t(i_1,\ldots,i_m) \leqslant t'(i_1,\ldots,i_m)$
then we must also have  the satisfaction
${\textswab M}(\Lambda,p)\models
t(\underline{{i_1}}/p,\ldots,\underline{{i_m}}/p) \leqslant
t'(\underline{{i_1}}/p,\ldots,\underline{{i_m}}/p).$
\end{lemma}
\begin{proof} Noting that
${\rm Q}\vdash\forall x,y\big(x\leqslant y\leftrightarrow \exists z
(x+z=y)\big)$ by the assumption there exists an $\beta\!\!\in\!\!{\mathcal M}$
such that ${\mathcal M}\models t(i_1,\ldots,i_m)+\beta=t'(i_1,\ldots,i_m)$.
On the other hand ${\mathcal M}\models\beta\leqslant
t'(i_1,\ldots,i_m)$, so by Lemma~\ref{lemm2} there exist a term $u$
and some $j_1,\ldots,j_k\leqslant\omega_1(\alpha)$ such that ${\mathcal M}\models
\beta=u(j_1,\ldots,j_k)$. Thus the equality
$t(i_1,\ldots,i_m)+u(j_1,\ldots,j_k)=s(i_1,\ldots,i_m)$ holds in
${\mathcal M}$. Now, by Lemma \ref{lem4}, \newline\centerline{${\textswab M}(\Lambda,p)\models
t(\underline{{i_1}}/p,\ldots,\underline{{i_m}}/p)+u(\underline{{j_1}}/p,\ldots,\underline{{j_k}}/p)=
t'(\underline{{i_1}}/p,\ldots,\underline{{i_m}}/p)$,} whence
$t(\underline{{i_1}}/p,\ldots,\underline{{i_m}}/p)\leqslant
t'(\underline{{i_1}}/p,\ldots,\underline{{i_m}}/p)$ is satisfied  in ${\textswab M}(\Lambda,p)$.
\end{proof}

\begin{lemma}\label{lem6}
Suppose   $t(x_1,\ldots,x_m),t'(x_1,\ldots,x_m)$ are two
${\mathcal L}_A-$terms and $i_1,\ldots,i_m\leqslant \alpha^k$ are  elements of
${\mathcal M}$ for some standard number  $k\!\in\!\mathbb{N}$.
If it is true that ${\mathcal M}\models t(i_1,\ldots,i_m) \not= t'(i_1,\ldots,i_m)$
then
${\textswab M}(\Lambda,p)\models
t(\underline{{i_1}}/p,\ldots,\underline{{i_m}}/p) \not=
t'(\underline{{i_1}}/p,\ldots,\underline{{i_m}}/p)$ must be true too.
And if ${\mathcal M}\models t(i_1,\ldots,i_m) \not\leqslant t'(i_1,\ldots,i_m)$
then
${\textswab M}(\Lambda,p)\models
t(\underline{{i_1}}/p,\ldots,\underline{{i_m}}/p) \not\leqslant
t'(\underline{{i_1}}/p,\ldots,\underline{{i_m}}/p).$
\end{lemma}
\begin{proof}
It follows from Lemma \ref{lem5} (and Remark \ref{remark1}) noting that $p$ is an ${\rm I\Delta_0}-$evaluation on $\Lambda^{\langle\infty\rangle}$ and

\noindent ${\rm I\Delta_0}\vdash\forall x,y\big(x\neq y\longleftrightarrow{\mathfrak s} y\leqslant
x \, \vee{\mathfrak s} x\leqslant y\big)$, and

\noindent ${\rm I\Delta_0}\vdash \forall x,y\big(x\nleq y\longleftrightarrow{\mathfrak s} y\leqslant x\big)$.
\end{proof}

\begin{theorem}\label{open}
Suppose   $\psi(x_1,\ldots,x_m)$ is an open RNNF
${\mathcal L}_A-$formula and $i_1,\ldots,i_m\leqslant \alpha^k$ are  elements of
${\mathcal M}$ for some standard number  $k\!\in\!\mathbb{N}$.
    If  we have ${\mathcal M}\models \psi(i_1,\ldots,i_m) $
then we also have
${\textswab M}(\Lambda,p)\models
\psi(\underline{{i_1}}/p,\ldots,\underline{{i_m}}/p).$
\end{theorem}
\begin{proof}
Lemmas \ref{lem4}  and \ref{lem5} prove the theorem for atomic formulas,
and Lemma \ref{lem6} proves it for negated atomic formulas. For the
disjunctive and conjunctive compositions of those formulas one can prove the theorem by a simple
induction.
\end{proof}

\begin{theorem}\label{bounded}
Suppose   that $\varphi(x_1,\ldots,x_m)$ is a bounded
${\mathcal L}_A-$formula and that $i_1,\ldots,i_m\leqslant \alpha^k$ are  elements of
${\mathcal M}$ for some standard number  $k\!\in\!\mathbb{N}$.
If ${\mathcal M}\models \varphi(i_1,\ldots,i_m) $
then
${\textswab M}(\Lambda,p)\models
\varphi(\underline{{i_1}}/p,\ldots,\underline{{i_m}}/p).$
\end{theorem}
\begin{proof}
Every bounded formula can be written as an (equivalent) RNNF
formula. By Lemma~\ref{lemm2} the range of bounded quantifiers of a
formula whose all  parameters belong to the set
\newline\centerline{$\{t(i_1,\ldots,i_m)\mid i_1,\ldots,i_m\leqslant \alpha \ \& \  t \ {\rm is \
an} \ {\mathcal L}_A-{\rm term}\}$} is indeed that set again. Now the
conclusion follows from Theorem \ref{open}.

\noindent $\blacktriangleright$ {\em An alternative proof:} To make
this important theorem more clear, we sketch another proof, which is not really too different but has more
model-theoretic flavor. Consider the above set again
$$\langle[0,\alpha]\rangle_{\mathcal M}=\{t(i_1,\ldots,i_m)\mid i_1,\ldots,i_m\leqslant \alpha \
\& \  t \ {\rm is \ an} \ {\mathcal L}_A-{\rm term}\}$$ which is a subset
of ${\mathcal M}$ closed under the successor, addition, and multiplication,
and thus forms a submodel of ${\mathcal M}$ (generated by
$[0,\alpha]=\{x\!\in\!{\mathcal M}\mid x\leqslant \alpha\}$). This submodel is an initial
segment of ${\mathcal M}$ by Lemma~\ref{lemm2}. Hence, whenever
${\mathcal M}\models\varphi$, for a bounded formula $\varphi$ with parameters in
$[0,\alpha]$, then $\langle[0,\alpha]\rangle_{\mathcal M}\models\varphi$.

Now, similarly, the
set
$$\langle[\underline{{0}}/p,\underline{{\alpha}}/p]\rangle_{\mathcal N}=
\{t(\underline{{i_1}}/p,\ldots,\underline{{i_m}}/p)\mid i_1,\ldots,i_m\leqslant \alpha
\ \& \  t \ {\rm is \ an} \ {\mathcal L}_A-{\rm term}\}$$ is an initial
segment and a submodel of ${\mathcal N}={\textswab M}(\Lambda,p)$. Thus if
$\langle[\underline{{0}}/p,\underline{{\alpha}}/p]\rangle_{\mathcal N}\models\varphi$, where
$\varphi$ is a bounded formula with parameters in
$[\underline{{0}}/p,\underline{{\alpha}}/p]$, then ${\textswab M}(\Lambda,p)\models\varphi$. Finally, we note that the mapping
$t(i_1,\ldots,i_m)\mapsto t(\underline{{i_1}}/p,\ldots,\underline{{i_m}}/p)$
defines a bijection between $\langle[0,\alpha]\rangle_{\mathcal M}$ and
$\langle[\underline{{0}}/p,\underline{{\alpha}}/p]\rangle_{\mathcal N}$ which is also an
isomorphism by Lemmas
 \ref{lem4}, \ref{lem5} and \ref{lem6}. So the proof of the theorem goes as
follows:

If ${\mathcal M}\models\varphi(i_1,\ldots,i_m)$  then
$\langle[0,\alpha]\rangle_{\mathcal M}\models\varphi(i_1,\ldots,i_m)$,  so
$\langle[\underline{{0}}/p,\underline{{\alpha}}/p]\rangle_{\mathcal N}\models\varphi(\underline{{i_1}}/p,\ldots,\underline{{i_m}}/p)$
hence ${\textswab M}(\Lambda,p)\models\varphi(\underline{{i_1}}/p,\ldots,\underline{{i_m}}/p)$.
\end{proof}

\begin{corollary}\label{last}
By the above assumptions, ${\textswab M}(\Lambda,p)\models\theta(\underline{\alpha}/p)$.\hfill
$\subset\!\!\!\!\supset$
\end{corollary}

Let us summarize what was argued in the last few pages.

\medskip

\begin{proof} {\bf (Of Theorem \ref{O1H}.)}
By the assumption of the theorem,  the theory $({\rm I\Delta_0+\Omega_1})+\exists x\!\in\!\ell\!{\it og}^2\theta(x)+{\rm HCon}({\rm I\Delta_0+\Omega_1})$ is consistent. So there is a model
$${\mathcal M}\models
({\rm I\Delta_0+\Omega_1})+\big(\alpha\!\in\!\ell\!{\it og}^2\wedge\theta(\alpha)\big)+{\rm HCon}({\rm I\Delta_0+\Omega_1}),$$ where
$\alpha\!\in\!{\mathcal M}$. We wish to show the consistency of  $({\rm I\Delta_0+\Omega_1})+\exists x\!\in\!\ell\!{\it og}^3\theta(x)$ by constructing another model $${\mathcal N}\models({\rm I\Delta_0+\Omega_1})+\exists x\!\in\!\ell\!{\it og}^3\theta(x).$$
If $\alpha$ is standard (i.e., $\alpha\!\in\!\mathbb{N}$) then one can take ${\mathcal N}={\mathcal M}$. But if $\alpha\!\in\!{\mathcal M}$ is non-standard, then we proceed as follows:
Take $\Lambda$ to be the following
set of terms: $\Lambda=\{\underline{j}\mid j\leqslant\omega_1(\alpha)\}\cup\{{\sf w}_j\mid j\leqslant\alpha\}$ in which the terms $\underline{j}$'s and ${\sf w}_j$'s are defined inductively as  $\underline{0}=0$,
$\underline{j+1}={\mathfrak s} \underline{j}$; and ${\sf w}_0=\underline{4}$, ${\sf
w}_{j+1}={\mathfrak w}({\sf w}_j)$. Here ${\mathfrak s}$ is the successor function, and  ${\mathfrak w}$ denotes the Skolem function symbol ${\mathfrak f}_{\exists y\left(y=\omega_1(x)\right)}$. Now
$\omega_2(\ulcorner\Lambda\urcorner)$ is of order (far less than) $2^{2^\alpha}$ which
exists  by the assumption ${\mathcal M}\models \alpha\!\in\!\ell\!{\it og}^2$.  Then by
Theorem \ref{log4} for a non-standard $j$ the set of terms
$\Lambda^{\langle j\rangle}$ has a code in ${\mathcal M}$. Thus the assumption ${\mathcal M}\models{\rm HCon}({\rm I\Delta_0+\Omega_1})$ implies that there must exists an $({\rm I\Delta_0+\Omega_1})-$evaluation $p$ on $\Lambda^{\langle j\rangle}$.  Then one can
form the model ${\mathcal N}={\textswab M}(\Lambda,p)$. Now ${\mathcal N}\models{\rm I\Delta_0+\Omega_1}$ by Lemma \ref{lm:mlambdap}, and also ${\mathcal N}\models \underline{\alpha}/p\!\in\!\ell\!{\it og}^3$ by the definition of ${\sf w}_\alpha$. Finally, ${\mathcal N}\models\theta(\underline{\alpha}/p)$ by Corollary \ref{last} (of Theorem \ref{bounded}).  Whence ${\mathcal N}$ is a model of the theory $({\rm I\Delta_0+\Omega_1})+\exists x\!\in\!\ell\!{\it og}^3\theta(x)$; and this finishes the proof of its consistency.
\end{proof}


\section{Herbrand Consistency of ${\rm I\Delta_0}$}
Our definition of Herbrand consistency is not best suited for ${\rm I\Delta_0}$: there are $\omega_1(\ulcorner\Lambda\urcorner)-$many evaluations on a given set of terms $\Lambda$. Though this may not seem a big problem in the first glance (one can change or modify the definition accordingly) but special care is needed for generalizing the results to the case of ${\rm I\Delta_0}$. In the first subsection we pinpoint the critical usages of ${\rm \Omega_1}$ and in the second subsection we tailor the definitions and theorems in a way that we can prove our main theorem for ${\rm I\Delta_0}$ finally.

\subsection{Essentiality of $\Omega_1$}
We made an essential use of $\Omega_1$ in the following parts of our arguments:

1- The totality of the $\omega_1$ function was needed for the upper bound of the code of  an evaluation on a given set of terms $\Lambda$. Namely, the code of any evaluation on $\Lambda$ is of order $\omega_1(\ulcorner\Lambda\urcorner)$, see Lemma \ref{codeomega1}.  And indeed there is no escape from this bound since, as it was explained after Lemma \ref{codeomega1}, there are $\exp(2|\Lambda|^2)$ evaluations on $\Lambda$, and if $|\Lambda|\approx\log\ulcorner\Lambda\urcorner$ then there could exist as many as $\omega_1(\ulcorner\Lambda\urcorner)^2$ evaluations on $\Lambda$. So, if $\Omega_1$ is not available, then there could be a large and non-standard  set of terms $\Gamma$ in a model ${\mathcal M}$ such that ${\mathcal M}$ cannot see all the evaluations on $\Gamma$. One of those evaluations could be a $T-$evaluation, that an end-extension of ${\mathcal M}$, say ${\mathcal K}$, can see. Then $\Gamma$ is a Herbrand proof of contradiction in ${\mathcal M}$ because in ${\mathcal M}$'s view there is no $T-$evaluation on $\Gamma$. But there could be indeed a very large $T-$evaluation on $\Gamma$ which ${\mathcal M}$ could not see, but ${\mathcal K}$ can. Thus the definition of ${\rm HCon}$ is deficient for ${\rm I\Delta_0}$ (where $\Omega_1$ is not there) and one cannot consider all the set of terms; those for which the $\omega_1$ of their codes exist, should be considered instead.

2- The second critical use of $\Omega_1$ was in the definition of ${\sf w}_j$'s for shrinking the (double-)logarithmic witness ${\mathcal M}\models\alpha\!\in\!\ell\!{\it og}^2$ to ${\mathcal N}\models{\sf w}_\alpha/p\!\in\!\ell\!{\it og}^3$. There we constructed the sequence $\langle {\sf w}_0,\ldots, {\sf w}_\alpha\rangle$ of terms such that ${\sf w}_0=\underline{4}$ and ${\sf w}_{j+1}={\mathfrak w}({\sf w}_j)$ where ${\mathfrak w}$ is the Skolem function symbol ${\mathfrak f}_{\exists y [y=\omega_1(x)]}$. And this was in our disposal because $\Omega_1=\forall x\exists y [y=\omega_1(x)]$ was one of the axioms (of ${\rm I\Delta_0+\Omega_1}$) and thus every $({\rm I\Delta_0+\Omega_1})-$evluation must have satisfied ${\mathfrak w}(t)=\omega_1(t)$.

Note that we also required $\Lambda$ to contain $\{\underline{j}\mid j\leqslant\omega_1(\alpha)\}$, but for this we  did not need the existence of $\omega_1(\alpha)$; it was guaranteed by the assumption ${\mathcal M}\models\alpha\!\in\!\ell\!{\it og}^2$.

\subsection{Tailoring for ${\rm I\Delta_0}$}
Here we introduce the necessary modifications on the above two points.

\subsubsection{The Definition of ${\rm HCon}^*$}

The first point can be dealt with by tailoring the definition of ${\rm HCon}$ for ${\rm I\Delta_0}$:

\begin{definition}{\rm
A theory $T$ is called Herbrand Consistent$^\ast$, denoted symbolically as ${\rm HCon}^*(T)$, when for all set of terms $\Lambda$, if  $\omega_1(\ulcorner\Lambda\urcorner)$ exists then  there is an $T-$evaluation on $\Lambda$.\hfill$\subset\!\!\!\!\supset$}
\end{definition}
This, obviously, can again be formalized in the language of arithmetic. The new definition cannot harm our arguments too much, because we needed ${\rm HCon}$ only for some special set of terms. And it was $\Lambda^{\langle j\rangle}$ for a non-standard $j$ where $\Lambda=\{\underline{j}\mid j\leqslant\omega_1(\alpha)\}\cup\{{\sf w}_j\mid j\leqslant\alpha\}$. For constructing the model ${\textswab M}(\Lambda,p)$ we already needed the existence of $\omega_2(\ulcorner\Lambda\urcorner)$ (see the beginning of the proof of Theorem \ref{O1H} before Lemma \ref{lemm1}).  Thus if we require the existence of $\omega_1(\ulcorner\Lambda\urcorner)$ in the definition of ${\rm HCon}^*$, then we will need the existence of $\omega_2(\ulcorner\Lambda\urcorner)$ later in the proof! Thus the first deficiency can be overcome.

\subsubsection{The Cuts ${\mathcal I}$ and ${\mathcal J}$}
In the absence of $\Omega_1$ we cannot define the above sequence
$\langle {\sf w}_0,\ldots, {\sf w}_\alpha\rangle$ satisfying ${\sf w}_{j+1}=\omega_1({{\sf w}_j})$. The most we can do inside ${\rm I\Delta_0}$ is to define a sequence like $\langle v_0,\ldots,v_\beta\rangle$ where $v_0=m$ and $v_{j+1}=(v_j)^n$ for some fixed $m,n\!\in\!\mathbb{N}$. Then $v_\beta=a^{n2^{\beta}}\leqslant{\mathcal P}(\exp^2(\beta))$. Thus we cannot get anything larger than $\exp^2$, and so for shortening a witness we should start from $\ell\!{\it og}$ and remain in the realm of $\ell\!{\it og}^2$. Indeed by the arguments of the beginning of the proof of Theorem \ref{O1H} before Lemma \ref{lemm1} we did not need the existence of $\exp^2(\alpha)$ for the existence of $\omega_2(\ulcorner\Lambda\urcorner)$. We needed only $\exp^2\big(4(\log\alpha)^4\big)$.  Thus it seems natural to consider the cut ${\mathcal I}=\{x\mid \exists y[y=\exp^2\big(4(\log\alpha)^4\big)]\}$ and its logarithm ${\mathcal J}=\{x\mid \exists y [y=\exp^2\big(4\alpha^4\big)]\}$.  We first note that Adamowicz's theorem (Theorem \ref{th2}) holds for ${\rm I\Delta_0}$ and any $n\!\in\!\mathbb{N}$; i.e., there exists a bounded formula whose $\ell\!{\it og}^n-$witness cannot {\em consistently} be shortened to $\ell\!{\it og}^{n+1}$. Indeed this theorem holds for any cut $I$ and its logarithm  which is definition the cut $J=\{x\mid \exists y[y=\exp(x)\wedge y\!\in\! I]\}$. The only relation between $\ell\!{\it og}^n$ and $\ell\!{\it og}^{n+1}$  needed in the proof of Theorem \ref{th1} is that $2^x\!\in\!\ell\!{\it og}^n \iff x\!\in\!\ell\!{\it og}^{n+1}$; see \cite{Ada02}. And the proof works for any cut $I$ and $J$ which satisfy $\forall x (2^x\!\in\! I\iff x\!\in\! J)$. The cuts ${\mathcal I}$ and ${\mathcal J}$ defined above satisfy this
 as well ($\exp(x)\!\in\!{\mathcal I}\iff x\!\in\!{\mathcal J}$). So, we repeat Theorem \ref{th1} as:
\begin{theorem}[\cite{Ada02}]\label{ij1}
There exists a bounded formula $\eta(\overline{x})$ such that the theory {${\rm I\Delta_0} \,+\, \exists \overline{x}\!\in\!{\mathcal I}\eta(\overline{x})$} is consistent, but  {${\rm I\Delta_0} \,+\, \exists \overline{x}\!\in\!{\mathcal J}\eta(\overline{x})$} is not consistent.\hfill $\subset\!\!\!\!\supset$
\end{theorem}

\subsubsection{The Main Theorem for ${\rm I\Delta_0}$}

Let us note that the following theorem together with Theorem \ref{ij1} prove that ${\rm I\Delta_0}\not\vdash{\rm HCon}^*({\rm I\Delta_0})$.

\begin{theorem}\label{ij2}
For any bounded formula $\theta(x)$, if the theory {${\rm I\Delta_0} \,+\, \exists
x\!\in\!{\mathcal I}\theta(x) \,+\, {\rm HCon}^*({\rm I\Delta_0})$} is consistent then so is  the theory {${\rm I\Delta_0} \,+\, \exists
x\!\!\in\!\!{\mathcal J}\theta(x)$.}
\end{theorem}

\begin{proof} Suppose the theory ${\rm I\Delta_0}+\exists x\!\in\!{\mathcal I}\theta(x)+{\rm HCon}^*({\rm I\Delta_0})$ is consistent. So there exists a model
$${\mathcal M}\models
{\rm I\Delta_0}+\big(\alpha\!\in\!{\mathcal I}\wedge\theta(\alpha)\big)+{\rm HCon}^*({\rm I\Delta_0}),$$ where
$\alpha\!\in\!{\mathcal M}$. We will show the consistency of  ${\rm I\Delta_0}+\exists x\!\in\!{\mathcal J}\theta(x)$ by constructing another model $${\mathcal N}\models{\rm I\Delta_0}+\exists x\!\in\!{\mathcal J}\theta(x).$$
If $\alpha$ is standard (i.e., $\alpha\!\!\in\!\!\mathbb{N}$) then one can take ${\mathcal N}={\mathcal M}$. But if $\alpha\!\!\in\!\!{\mathcal M}$ is non-standard, then we proceed as follows:

 Let $\Upsilon=\{0,  0+0,  0^2, {\mathfrak c}, {\mathfrak c}^2, {\mathfrak c}^2+0, {\mathfrak s}{\mathfrak c}, {\mathfrak q}{\mathfrak c},  ({\mathfrak s}{\mathfrak c})^2,  ({\mathfrak s}{\mathfrak c})^2+0\}$ where ${\mathfrak q}$ is the Skolem function symbol for the formula ${\exists y(y\leqslant x^2\wedge y=x^2)}$ and  ${\mathfrak c}$ is the Skolem constant symbol for the  sentence (see Example \ref{exampleB})
 \newline\centerline{${\exists x\big(\exists w (w\leqslant x^2\wedge w=x^2)\wedge\forall v (v\not\leqslant({\mathfrak s} x)^2\wedge v\not=({\mathfrak s} x)^2)\big)}$.} We can use the argument of Example \ref{exampleB}, since for the bounded formula  $\psi(x)=\exists y\leqslant x^2(y=x\cdot x)$,  the sentence $\textrm{ind}_\psi$ is an axiom of the theory ${\rm I\Delta_0}$.
Take
{$\Lambda=\Upsilon\cup\{\underline{j}\mid j\leqslant\omega_1(\alpha)\}\cup\{{\sf z}_j\mid j\leqslant4\alpha^4\}$} in which the terms $\underline{j}$'s and ${\sf z}_j$'s are defined inductively as  $\underline{0}=0$,
$\underline{j+1}={\mathfrak s} \underline{j}$; and ${\sf z}_0=\underline{2}$, ${\sf
z}_{j+1}={\mathfrak q}({\sf z}_j)$.   
Now
$\omega_2(\ulcorner\Lambda\urcorner)$ is of order  $\exp^2\big(4(\log\alpha)^4\big)$ which
exists  by the assumption ${\mathcal M}\models \alpha\!\in\!{\mathcal I}$.   Then by
Theorem \ref{log4} for a non-standard $j$ the set of terms
$\Lambda^{\langle j\rangle}$ has a code in ${\mathcal M}$. Thus the assumption ${\mathcal M}\models{\rm HCon}^*({\rm I\Delta_0})$ implies that there must exists an ${\rm I\Delta_0}-$evaluation $p$ on $\Lambda^{\langle j\rangle}$.  Then one can
form the model ${\mathcal N}={\textswab M}(\Lambda,p)$. Now ${\mathcal N}\models{\rm I\Delta_0}$ by Lemma~\ref{lm:mlambdap}, and also ${\mathcal N}\models \underline{\alpha}/p\!\in\!{\mathcal J}$ by the definition of ${\sf z}_{4\alpha^4}$ (which represents $\exp^2(4\alpha^4)$). Note that $p\models {\sf z}_{j+1}={\sf z}_j\cdot {\sf z}_j$ by the argument of Example~\ref{exampleB}, and also the code of the sequence $\langle {\sf z}_0,\ldots,{\sf z}_{4\alpha^4}\rangle$ is of order $\exp\big({(4\alpha^4)^2}\big)\leqslant\exp^2(4(\log\alpha)^4)$ which exists since $\alpha\!\in\!{\mathcal I}$. Finally, ${\mathcal N}\models\theta(\underline{\alpha}/p)$ by Corollary \ref{last} (of Theorem \ref{bounded}).  Whence ${\mathcal N}$ is a model of the theory ${\rm I\Delta_0}+\exists x\!\in\!{\mathcal J}\theta(x)$; what proves its consistency.
\end{proof}




\section{Conclusions}
An important property of Herbrand consistency of the theories ${\rm I\Delta_0+\Omega_1}$ and ${\rm I\Delta_0}$ has been proved. That property immediately implies G\"odel's second incompleteness theorem for the notion of Herbrand consistency in those theories. However, this version of G\"odel's theorem has come a long way. The original presumption of Paris \& Wilkie \cite{PaWi81} asked for a proof of ${\rm I\Delta_0}\not\vdash{\rm CFCon}({\rm I\Delta_0})$, without specifying any  variant of Cut-Free Consistency ${\rm CFCon}$: ``Presumably ${\rm I\Delta_0}\not\vdash{\rm CFCon}({\rm I\Delta_0})$ although we do not know this at present". Willard \cite{Wil02} solved this problem for the Tableau Consistency variant. Pudl\'ak \cite{Pud85} asked a more specific question: ``we know only that $T\not\vdash{\rm HCon}(T)$ for $T$ containing at least ${\rm I\Delta_0}+{\rm Exp}$, for weaker theories it is an open problem".   In \cite{Sal02} this problem was studied for the theories ${\rm I\Delta_0+\Omega_1}$ and ${\rm I\Delta_0}$ (and a theory in between these two, namely ${\rm I\Delta_0}$ plus the totality of the $x\mapsto x^{\log^2 x}$ function). The proof of ${\rm I\Delta_0+\Omega_1}\not\vdash{\rm HCon}({\rm I\Delta_0+\Omega_1})$ given here was presented for the first time in Chapter 5 of \cite{Sal02}. But the unprovability of ${\rm HCon}({\rm I\Delta_0})$ in ${\rm I\Delta_0}$ was not as easy as it would have seemed. In Chapter 3 of \cite{Sal02} this unprovability was proved for a re-axiomatization of ${\rm I\Delta_0}$.

 Our reason for using the induction formula $\textrm{ind}_\psi$, where $\psi(x)$ is
 the bounded formula $\exists y\leqslant x^2(y=x^2)$, was
 having an Skolem function symbol for squaring ${\mathfrak q}(x)=x^2$.
This way the G\"odel code of ${\mathfrak q}(x)$ is $M\cdot \ulcorner x\urcorner$ for a fixed $M\!\in\!\mathbb{N}$, and thus the code of ${\mathfrak q}^n(x)$ is $M^n\cdot \ulcorner x\urcorner$ which is of order $\exp(n)$. So, we could code a term representing the number $x^{\exp(n)}$ (=${\mathfrak q}^n(x)$) by a number of order $\exp(n)$. But if we coded the number $x^{\exp(n)}$ directly, that would be the code of ${x\cdot x\cdot \ldots \cdot x}$ (with $2^n-\text{times} \ x$) which is of order $(\ulcorner x\urcorner)^{2^n}$ or $\exp^2(n)$.  In that case, the code of the sequence $\langle {\sf z}_0,\ldots,{\sf z}_{4\alpha^4}\rangle$ would be of order $\exp^2((4\alpha^4)^2)$, but we used the order $\exp\big({(4\alpha^4)^2}\big)$ in the proof of Theorem \ref{ij2} (since we had at most $\exp^2(4(\log\alpha)^4)$ in our disposal  - which is far less than $\exp^2((4\alpha^4)^2)$). That way, we avoided accepting the totality of the squaring function ${\rm \Omega_0}: \forall x\exists y (y=x\cdot x)$ as an (additional) axiom.

This point deserves another look: define the terms $\{{\sf z}_i\}$, $\{{\sf u}_i\}$, and $\{{\sf v}_i\}$ inductively as ${\sf z}_0=2$, ${\sf z}_{j+1}={\mathfrak q}({\sf z}_j)$; ${\sf u}_0=2$, ${\sf u}_{j+1}=({\sf z}_j)^2$; and ${\sf v}_0=2$, ${\sf v}_{j+1}=({\sf v}_j)^2$. Then the codes of the terms ${\sf z}_n$'s and ${\sf v}_n$'s are of order $\mathcal{P}(2^n)$, but the code of ${\sf v}_n$'s are of order $\mathcal{P}(2^{2^n})$. On the other hand, the terms ${\sf z}_i$, ${\sf u}_i$ and ${\sf v}_i$ have the same value ($2^{2^i}$) in any model of ${\rm I\Delta_0}$. In fact, for $i\leqslant\omega_1(\alpha)$ we have ${\sf z}_i\!\in\!\Lambda$ and also ${\sf u}_i\!\in\!\Lambda^{\langle 1\rangle}$; but ${\sf v}_i$'s are too big to fit in small sets of terms.

\bigskip

 Our treatment of G\"odel's second incompleteness theorem for Herbrand consistency in weak arithmetics,  can be summarized in the following improvements to the classical treatments (cf. the first paragraph of Appendix E in \cite{Wil07}):

 \noindent (1) For Skolemizing a formula we did not transform it to a prenex normal form. This allowed a more efficient Skolemization and Herbrandization of formulas.

\noindent  (2) Propositional satisfiability was achieved by evaluations, which are partial (Herbrand) models; see also \cite{Ada01,Ada02,AZ07,Kol06,Sal01,Sal02}.

 \noindent (3) For logarithmic shortening of bounded witnesses in ${\rm I\Delta_0}$, we could not go from $\ell\!{\it og}$ to $\ell\!{\it og}^2$ directly. Instead we used the condition $\omega_1^2(x)^4\!\in\!\ell\!{\it og}$ (equivalently $x\!\in\!{\mathcal I}$) to get to $4x^4\!\in\!\ell\!{\it og}^2$ (equivalently $x\!\in\!{\mathcal J}$). For that we used the improved version of Adamowicz's theorem \cite{Ada02} (Theorem \ref{ij1}).

 \noindent (4) And finally, we used the trick of $\textrm{ind}_\psi$ to get an Skolem function symbol for the squaring function. Ideally, one would not use any induction axiom for proving a formula like ${\rm \Omega_0}: \ \forall x\exists y (y=x^2)$. This is an ${\rm Q}-$derivable sentence, and adding it as an axiom seems much more natural than proving it by an inductive argument.
  But, fortunately, there was a way of avoiding the acceptance of ${\rm \Omega_0}$ as an axiom, and that was proving its $\Pi_1-$equivalent $\forall x\exists y\leqslant x^2(y=x^2)$ by induction on its bounded part $\exists y\leqslant x^2(y=x^2)$ (see Example \ref{exampleB} and the proof of Theorem~\ref{ij2}). That induction axiom could give us a free Skolem function symbol for the squaring operation, provided that we did not prenex normalize the induction axiom, and instead Skolemize it more effectively $-$ see point (1) above. Prenex normalizing and then Skolemizing the induction axioms can be so cumbersome that many would prefer avoiding them, but accepting new axioms instead! Trying to prenex normalize the induction axiom $\textrm{ind}_\psi$ for $\psi=\exists y\leqslant x^2(y=x^2)$ in Example \ref{exampleB} can give a hint for its difficulty.

\bigskip

\noindent In the end, we conjecture that by using our coding techniques and definitions of Herbrand consistency,  the results of  L. A. Ko{\l}odziejczyk \cite{Kol06}   can be generalized for showing the following unprovability:

\begin{conjecture}{\rm
%
$ \bigcup_n({\rm I\Delta_0}+{\rm \Omega_n})\not\vdash{\rm HCon}^*({\rm I\Delta_0})$.
}\end{conjecture}
\begin{question}{\rm
Can a {\sc Book} proof (in the words of Paul Erd\"{o}s) be given for G\"odel's second incompleteness theorem $T\not\vdash\mathcal{H}\textswab{C}\textswab{o}\textswab{n}(T)$ for any theory $T\supseteq {\rm Q}$ and a canonical  definition of Herbrand consistency $\mathcal{H}\textswab{C}\textswab{o}\textswab{n}$?
}\end{question}


\subsubsection*{Acknowledgements} This research was partially supported by the grant
${\sf N}^{\underline{\sf o}}\,{86030011}$ of the
Institute for Studies in Theoretical Physics and Mathematics
$\bigcirc\!\!\!\!/\!\!\!\bullet\!\!\!/\!\!\!\!\bigcirc$ $\mathbb{I}\mathbb{P}\mathbb{M}$,
Niavaran, Tehran, Iran.


\end{document}